\documentclass[10 pt]{amsart}

\usepackage{amsmath}
\usepackage{amsfonts}
\usepackage{amssymb}
\usepackage{amstext}
\usepackage{amsbsy}
\usepackage{amsopn}
\usepackage{amsthm}
\usepackage{amsxtra}
\usepackage{graphicx}
\usepackage{hyperref}
\usepackage{enumitem}
\usepackage{xcolor}
\usepackage{enumitem}   
\usepackage[normalem]{ulem}
\usepackage{soul}
\usepackage{tikz}
\usepackage[justification=centering]{caption}

\newtheorem{theorem}{Theorem}[section]
\newtheorem{lemma}[theorem]{Lemma}
\newtheorem{proposition}[theorem]{Proposition}

\newtheorem{corollary}[theorem]{Corollary}
\theoremstyle{definition}
\newtheorem{definition}[theorem]{Definition}

\newtheorem*{comment*}{Comment}

\newcommand{\R}{\mathbb{R}}
\newcommand{\Z}{\mathbb{Z}}

\newcommand{\N}{\mathbb{N}}

\newcommand{\Aut}{\mathrm{Aut}}

\newcommand{\CA}{\mathcal{A}}

\newcommand{\CF}{\mathcal{F}}

\newcommand{\CL}{\mathcal{L}}

\newcommand{\CP}{\mathcal{P}}

\newcommand{\one}{\boldsymbol{1}}

\newcommand{\aut}{\mathrm{Aut}}
\newcommand{\Id}{\mathrm{Id}}

\newcommand{\tp}{\mathrm{top}}

\newcommand{\interior}[1]{%
  {\kern0pt#1}^{\mathrm{o}}%
}

\title{Stable covers of subshifts}
\author{Solly Coles}
\address{Department of Mathematics, Tufts University, 177 College Avenue, Medford, MA 02155}
\author{Van Cyr}
\address{Department of Mathematics, Bucknell University, 1 Dent Drive, Lewisburg, PA 17837}
\author{Bryna Kra}
\address{Department of Mathematics, Northwestern University, 2033 Sheridan Road, 
Evanston, IL 60208}
\author{Ronnie Pavlov}
\address{Department of Mathematics, University of Denver, 2390 S. York Street, 
Denver, CO 80210}

\thanks{BK and RP gratefully acknowledge the support of the Simons Foundation and BK the support of NSF grant DMS-2348315.}
\begin{document}

\maketitle

\begin{abstract}
Given a dynamical system, a characteristic measure is a Borel probability measure invariant under all of its automorphisms.  Frisch and Tamuz asked if every symbolic system supports such a measure.  Motivated by this problem, we study the natural cover of a subshift by its shift of finite type approximations and two senses in which this  cover can be said to stabilize.  The first is in terms of entropy decay and the second in terms of periodic points.  We show that the first type of stabilization gives a new characterization of the class of language stable shifts and demonstrates that there is a mechanism for producing a characteristic measures that relies only on entropy differences.  For the second type of stabilization, we show that this defines a new class of subshifts, invariant under conjugacies, that have characteristic measures.
\end{abstract}

\section{Characteristic measures}

Given a symbolic dynamical system $(X,\sigma)$, there is a canonical sequence of approximations to $X$ by simpler systems: subshifts of finite type.  We study the interplay between the properties of $X$ and various senses in which this approximating sequence stabilizes.  Our motivation to study stabilizations is to make progress on the characteristic measure problem introduced by Frisch and Tamuz~\cite{FT}.  They define a measure to be {\em characteristic} for a topological system $(X,T)$ if it is invariant under the automorphism group $\Aut(X,T)$.  Not every topological system supports such a measure, for instance it is mentioned in~\cite{FT} that the identity map acting on the Cantor set does not.  There are even examples of minimal systems that fail to support a characteristic measure~\cite{FSZ}. 

On the other hand, there are many classes of symbolic systems known to support a characteristic measure. For instance, Parry~\cite{parry} shows that any mixing shift of finite type has a unique measure of maximal entropy, which therefore must be a characteristic measure. Though not phrased in this terminology, it follows from the Krylov-Bogolubov theorem that any system with an amenable automorphism group supports such a measure. It is easy to check that any symbolic system with a periodic point also has a characteristic measure. Frisch and Tamuz~\cite{FT} show that any symbolic system with zero entropy supports one. They also pose the general question~\cite[Question 1.3]{FT}: does every symbolic $\mathbb{Z}$-system support a characteristic measure?

A fruitful approach to this problem has been to make use of the {\em SFT cover} $\{X_n\}_{n\in\N}$ of a subshift $(X, \sigma)$.  This is a canonical sequence $(X_n,\sigma_n)$ 
of shifts of finite type such that   
\[
X = \bigcap_{n\in\N} X_n
\]
Using this cover, \cite{CK} introduces the class of language stable shifts, where the subshifts in the SFT cover stay constant for arbitrarily long runs.  Every language stable shift supports a characteristic measure~\cite{CK}. 
These shifts were further studied in a quantitative way in~\cite{CKP}, where it is shown that factors of language stable shifts have characteristic measures, provided the the approximating sequence in the SFT cover satisfies a well appoximability property. 
 
 In this work, we introduce two notions of stability for the SFT cover of a shift, that imply the existence of a characteristic measure.  Each class defined by these conditions contains all  language stable shifts, but new methods are used to show the existence of characteristic measures.   Define a shift $(X, \sigma)$ to be {\em period stable} if for all $m\in\N$, there exist $n, p\in\N$ such that the set of periodic points of minimal periodic $p$ in $X_n$ is the same as the set of periodic points of minimal periodic $p$ in $X_{n+m}$ (see Definition~\ref{def:periodstable}).  In Section~\ref{sec:periodic}, we show that this condition suffices for the existence of a characteristic measure. 
\begin{theorem}
\label{th:periodic}
    Every period stable subshift supports a characteristic measure. 
\end{theorem}
 
It is easy to check that any shift that has at least one periodic point is period stable, and so the theorem is most interesting in the case that $X$ does not have any periodic points.  In this case, the terms of the SFT cover $\{X_n\}_{n=1}^{\infty}$ provide a ``halo of periodic points,'' that are not actually elements of $X$ but nevertheless allow us to find a characteristic measure for $X$. 
Furthermore, any language stable shift $X$ (see~\cite{CK}) has the property that for all $m\in\N$ there exists $n$ such that $X_{n+m}=X_n$ and so in particular, any such shift is also period stable.  We show that our result goes beyond the cases of a shift with periodic points or a language stable shift, as both such classes are already known to support a characteristic measure. 
In Section~\ref{sec:construction},  we build an example to show that the class of period stable shifts contains a subshift that has no periodic points and is not language stable.

In~\cite{CK}, it is also shown that all language stable shifts have a characteristic measure that is a measure of maximal entropy.  Though we show that all period stable shifts also have a characteristic measure, the measure we produce is not necessarily one of maximal entropy.

Turning to a different type of stability, we consider of the entropy of shifts in the the SFT cover.  We define the class of {\em entropy stable shifts} in Section~\ref{sec:entropy}, and show that again this type of stability gives rise to a characteristic measure. 
\begin{theorem}
\label{th:etrnopy}
    Every topologically mixing entropy stable subshift supports a characteristic measure. 
\end{theorem}

In Section~\ref{sec:same}, we show that the class of entropy stable shifts is, in fact, a new characterization of the class of language stable shifts.  Although these two classes coincide, the mechanism used to show that entropy stable shifts have characteristic measures depends only on the entropy drops in terms in the SFT cover.  While not providing a new class of shifts not previously known to have characteristic measures, the method used in~\cite{CK} fundamentally requires a symbolic structure for showing language stable shifts have characteristic measures: the proof relies in a crucial way on the fact that automorphisms of language stable shifts are defined by block codes and extend to automorphisms of all but finitely many terms in the SFT cover of a language stable shift. 
In contrast, our proof that entropy stable shifts have characteristic measures only makes use of the entropy drop between runs of consecutive terms in the SFT cover.  Thus our method may well extend to some non-symbolic systems, giving a way to study more general topological dynamical systems 
using a natural cover by simpler systems.

\section{Notation and background}
\subsection{Subshifts and the language}
Assume that $\CA = \CA(X)$ is a finite alphabet and write $x\in \CA^\Z$ as $x = \bigl(x(n):n\in\Z\bigr)$. 
The {\em left shift} $\sigma\colon \CA^\Z\to \CA^\Z$ is 
defined by $(\sigma x)(n)  = x(n+1)$ for all $n\in\Z$. 
If $X\subseteq\mathcal{A}^{\Z}$ is nonempty, closed and $\sigma$-invariant, then we say that $(X, \sigma)$ is a {\em subshift}. 

For each word $w\in\mathcal{A}^*$, let $|w|$ denote the length of $w$, let $w_i$ denote the $i^{th}$ letter of $w$ ($0\leq i<|w|$), and let 
$$ 
[w]:=\{x\in\mathcal{A}^{\Z}\colon x(i)=w_i \text{ for all }0\leq i<|w|\} 
$$ 
be the cylinder set determined by $w$.  

In general, we use the convention that a subscript $w_i$ to denote the $i^{th}$ letter in the finite word $w$ and $x(i)$ to denote the $i^{th}$ letter in the element (infinite word)  $x\in X$.  However, when elements of $X$ have indices, we abuse the notation and combine these (such as in Section~\ref{subsec:W1}), but this should be clear from context. 

For a subshift $(X,\sigma)$, for each $n\in\N$ we write 
$$ 
\mathcal{L}_n(X):=\{w\in\mathcal{A}^n\colon[w]\cap X\neq\emptyset\}
$$ 
for the words of length $n$ in the {\em language}  $\mathcal{L}(X):=\bigcup_{n=1}^{\infty}\mathcal{L}_n(X)$ of $X$.  
We refer to any word in the language of $X$ as an {\em admissible} word.  If $a, u, b\in \CL(X)$ and $w= aub\in\CL(X)$, with the convention that one of $a$ and $b$ may be empty, we refer to $u$ as a {\em subword} of $w$, and we refer to $a$ as a {\em prefix} of the word $w$ and $b$ as a {\em suffix} of the word $w$.
If $w\in\CL(X)$, we write $w^\infty$ for the infinite periodic word $www\dots$ to the right and $^\infty w$ for the infinite periodic word $\dots www$ to the left.  

If $\mathcal{F}\subseteq\mathcal{A}^*$ then 
$$
X_{\mathcal{F}}:=\{x\in\{0,1\}^{\Z}\colon\sigma^i(x)\notin[w]\text{ for all }i\in\Z\text{ and }w\in\mathcal{F}\}
$$
is a subshift.  Any subshift $(X, \sigma)$ can be defined by specifying its language $\CL(X)$, or equivalently by specifying its {\em canonical set of forbidden words} $\mathcal{F}(X):=\mathcal{A}^*\setminus\mathcal{L}(X)$ because $X=X_{\mathcal{F}(X)}$.  In general, however, there could be a set $\mathcal{F}\subseteq\mathcal{A}^*$ such that $X=X_{\mathcal{F}}$ but $\mathcal{F}\neq\mathcal{F}(X)$.  In this case, we say that $X$ {\em can be defined with $\mathcal{F}$}.  A subshift $(X, \sigma)$ is a {\em shift of finite type} if it can be defined with a finite set $\mathcal{F}$.  We say that a forbidden word is {\em minimal} if it has no proper subword which is a canonical forbidden word, and note that the set of minimal forbidden words is a canonical way to define a subshift.  
A shift of finite type is called {\em nearest-neighbor} if it is defined with a finite set $\mathcal{F}$ of forbidden words all of which have length at most $2$.

\subsection{The SFT cover of a subshift}
For a subshift $(X, \sigma)$ with forbidden words $\mathcal{F}(X)$, for each $n\in\N$ let $\mathcal{F}_n(X):=\mathcal{A}^n\setminus\mathcal{L}_n(X)$  denote the forbidden words of length $n$ in $X$.  Define $X_n$ to be the subshift of finite type in $\mathcal{A}^{\Z}$ whose set of forbidden words is $\bigcup_{k=1}^n\mathcal{F}_k(X)$.  The sequence $\{X_n\}_{n=1}^{\infty}$ is called the {\em SFT cover} of $X$, and it follows immediately from the definitions that 
$$ 
X_1\supseteq X_2\supseteq X_3\supseteq\dots\supseteq X_n\supseteq X_{n+1}\supseteq\dots 
$$ 
and the reader can check that $X=\bigcap_{n=1}^{\infty}X_n$.  

A subshift $(X, \sigma)$ is {\em mixing} if for all words $u,v\in\CL(X),$ 
there is some $N\in\N$ such that for all $n\geq  N$ there is some word $w\in \CL_n(X)$ such that $uwv\in\CL(X)$. 
When the subshift  $(X, \sigma)$ is mixing, then 
each $X_n$  in its SFT cover is also mixing. In that case, it follows from Parry~\cite{parry} that each $X_n$ has a unique measure of maximal entropy $\mu_n$.

\subsection{The automorphism group of a subshift}
Let $\varphi\in\Aut(X)$ be an automorphism of the subshift $(X, \sigma)$, meaning that $\varphi\colon X\to X$ is a homeomorphism and $\varphi\circ\sigma = \sigma\circ\varphi$.  Let $\Aut(X) = \Aut(X, \sigma)$ denote the group of all automorphisms of the subshift $(X, \sigma)$.  

Given an automorphism $\varphi$ of the subshift $(X, \sigma)$, by the Curtis-Hedlund-Lyndon theorem there exists $R\in\N$ and a map $\Phi\colon\mathcal{L}_{2R+1}(X)\to\mathcal{A}$ such that 
$$ (\varphi(x))(i)=\Phi(x(i-R)x(i-R+1)\dots x(i) \dots x(i+R-1)x(i+R))$$
for all $x\in X$ and all $i\in\Z$ (note that we have added commas in the word in $x$ of length $2R+1$ for clarity). 
We call $\Phi$ a {\em block map implementing $\varphi$},  $R$ the {\em range} of the block map, and $\varphi$ a {\em sliding block code} with range $R$. 
The parameter $R$ is not uniquely defined, and if $\varphi$ has range $R$ then it also has range $R'$ for any $R'\geq R$.  

Let $\Aut_R(X)$ denote  the set of all $\varphi\in\Aut(X)$ with the property that both $\varphi$ and $\varphi^{-1}$ are sliding block codes of range $R$.  It follows immediately from the definitions that  $\Aut(X)=\bigcup_{R=0}^{\infty}\Aut_R(X)$ and $\Aut_R(X)\subseteq\Aut_{R+1}(X)$ for all $R\geq 1$.  We note the following elementary lemma for use in Section~\ref{sec:periodic}. 

\begin{lemma}
    Assume that $\varphi\in\Aut(X)$ has range $R$ and that $\Phi,\Phi^{-1}\colon\mathcal{L}_{2R+1}(X)\to\mathcal{A}$ 
    are range $R$ block maps that implement $\varphi, \varphi^{-1}$, respectively.  Then $\Phi^{-1}\circ\Phi$ implements the identity map $\Id\colon X\to X$ as a range $2R$ block code. 
\end{lemma}
\begin{proof}
   Assume that $\{X_n\}_{n=1}^\infty$ is the SFT cover of $X$. 
   For any $n\geq2R+1$, the map $\Phi$ (analogously for the $\Phi^{-1}$)  
   defines a map whose domain is $X_n$  by applying the block map $\Phi$ as a sliding block code to the elements of $X_n$.  By an  abuse of notation, we use $\varphi$ to denote the resulting automorphism of $X_n$.  A priori, this extends the domain of $\varphi\colon X\to X$ to a new map $\varphi\colon X_n\to\mathcal{A}^*$ for any $n\in\N$, with the property that $\varphi$ maps elements of $X$ to elements of $X$.  
   However, since  $\varphi(X)=X$, we claim that if $k\in\N$ is  fixed and $w\in\mathcal{L}_{2R+k}(X)$ is a fixed word, then the word $\Phi(w)\in\mathcal{A}^k$ given by 
$$ 
(\Phi(w))(j)=\Phi(w_j,w_{j+1},\dots,w_{j+R},\dots,w_{j+2R-1},w_{j+2R})
$$ 
 for each $0\leq j<k$ is  an element of $\mathcal{L}_k(X)$.  
 Namely, for each $w\in\mathcal{L}_{2R+k}(X)$, 
 there exists some $x\in[w]\cap X$ and the word $\Phi(w)$ occurs as a subword of $\varphi(x)\in X$.  In particular, $\varphi(X_{2R+k})\subseteq X_k$.  For $k\geq2R+1$, this allows us to apply $\Phi^{-1}$ to $\Phi(w)$ and obtain an element of $\mathcal{L}_{k-2R}(X)$.  Since $\Phi$ and $\Phi^{-1}$ are block codes that implement inverse maps from $X$ to $X$, it follows that $(\Phi^{-1}\circ\Phi)(w)_j=w_j$ for all $4R<j<|w|-4R$.  In other words,  $\Phi^{-1}\circ\Phi$ implements the identity map as a range $2R$ block code: if $k\geq2R+1$ and we apply $\Phi^{-1}\circ\Phi$  to a word in $\mathcal{L}_{2R+k}(X)$, the map removes the rightmost and leftmost $2R$ letters from the word and leaves the middle of the unchanged.
\end{proof}

\subsection{Measures and entropy}
If $(X, \sigma)$ is a subshift, we let $h_\tp(X)$ denote the topological entropy of $X$ and if $\mu$ is a $\sigma$-invariant measure on $X$, we let $h_\mu$ denote the entropy of the measure $\mu$. 

The measure $\mu$ is a {\em characteristic measure} for the subshift $(X, \sigma)$ if $\mu$ is invariant under the automorphism group $\aut(X, \sigma)$.

\section{Entropy stability}
\label{sec:entropy}

\subsection{Quantifying a stability condition}
We start by giving a condition on the language of a subshift that suffices to ensure that high entropy measures are close for words up to a fixed length (this property is often called effective intrinsic ergodicity in the literature).

\begin{theorem}
\label{th:main}
    Let $(X,\sigma)$ be a mixing shift of finite type with measure of maximal entropy $\mu$. Let $F$ be the maximal length of a minimal forbidden word  and let $s$ be the number of admissible words of length $F$. Let $\varepsilon>0$ and $\ell\in \N$. Suppose $\mu'$ is an invariant probability measure on $(X,\sigma)$
    satisfying 
    \begin{equation*}
    h_{\mu'}>h_{\tp}-\frac{\varepsilon^2}{(\frac{4}{3})^{2\ell} (\xi(s))^2},
    \end{equation*}
    where 
    \begin{equation}
    \label{eq:xi}
    \xi(s)=\frac{30s^{3(s^2+1)}}{1-\left(1-\frac{1}{4s^{2s^2}}\right)^{\frac{1}{s^2}}}.
    \end{equation}
    Then for all admissible words $w$ of length at most $\ell$, we have $|\mu'([w])-\mu([w])|<\varepsilon.$
\end{theorem}

We note that the function $\xi(s)$ is positive and increasing in $s.$

Our proof makes use of a method of Kadyrov~\cite{kadyrov} introduced to give an quantitative version on the existence  of a unique measure of maximal entropy for a nearest-neighbor subshift of finite type.  We start by introducing some notation and results from~\cite{bowen} and~\cite{Parry-Pollicott} on the spectral properties  of the Ruelle transfer operator.  Given a mixing shift of finite type $(X, \sigma)$, 
define $(X^+, \sigma)$ to be the one-sided subshift version of $X$.  Letting $C(X^+)$ denote the continuous real valued functions on $X^+$, define the {\em Ruelle transfer operator} 
$\CL\colon C(X^+)\to C(X^+)$  by $$\CL f(x)=\sum_{\sigma y=x}f(y).$$
There is a simple maximal positive eigenvalue $\lambda$ of $\CL$, with corresponding strictly positive eigenfunction $h.$ (Here we follow the convention of Bowen that $h$ denotes this eigenfunction and should not be confused with the notation for entropy throughout  this work.) The eigenvalue $\lambda$ also satisfies
$h_\tp(X^+)=\log\lambda.$
There is a unique probability measure $\nu$ satisfying $\CL^*\nu=\lambda\nu$, 
and without loss we further assume that $h$ satisfies $\nu(h)=1$.

The normalized transfer operator $\CL_0\colon C(X^+)\to C(X^+)$ is given by $$\CL_0=\frac{1}{\lambda}\Delta_h^{-1}\CL\Delta_h,$$ where $\Delta_h\colon C(X^+)\to C(X^+)$ is the operator $f\mapsto hf.$ The spectrum of $\CL_0$ is the spectrum of $\CL$ scaled by a factor of $\lambda$, and the constant functions are the eigenfunctions corresponding to the maximal eigenvalue. The corresponding eigenmeasure is the measure of maximal entropy $\mu$ of $(X^+, \sigma)$ and is given by $d\mu=h\, d\nu.$

In the next proposition, we  assume that $X^+$ is a nearest-neighbor shift of finite type.  We  later recode general shifts of finite type as nearest-neighbor shifts of finite type in the proof of Theorem~\ref{th:main}.

\begin{proposition}\label{prop:bounds}
  Let $(X^+,\sigma)$ be a mixing nearest-neighbor shift of finite type with measure of maximal entropy $\mu$. Let $w$ be an admissible word of length $\ell\geq 1$ and let  $f=\one_{[w]}$ and $g=f-\mu(f).$ There exist $A>0$ and $\beta\in (0,1)$ such that for all $n\geq 0$, 
  \begin{equation}
  \label{eq:bound-L}
  \|\CL_0^ng\|_\infty\leq \frac{A\|h\|_\infty}{\inf h}\beta^{n-\ell}, 
  \end{equation}
 where $\CL_0$ is the normalized transfer operator and $h$ is the   eigenfunction associated to  $\CL$. 
 Furthermore, letting $s$ denote the size of the alphabet, we have $\frac{\|h\|_\infty}{\inf h}\leq s^{s}$, and we can choose $A\leq 15s^{2s^2}$ and $$\frac{3}{4}\leq \beta\leq \left(1-\frac{1}{4s^{2s^2}}\right)^{\frac{1}{s^2}}.$$
\end{proposition}

\begin{proof}
Let $\CL\colon C(X^+)\to C(X^+)$ denote the Ruelle transfer operator, $\lambda$ its (simple and positive) maximal eigenvalue with eigenfunction $h$, and $\nu$ be the unique probability measure such that $\CL^*\nu = \lambda\nu$.
By~\cite[Lemma 1.12]{bowen}, there exist  $A>0$ and $\beta\in (0,1)$ such that for all $n\geq \ell$,
$$\|\tfrac{1}{\lambda^n}\CL^n(hg)\|_\infty\leq A\nu(|hg|)\beta^{n-\ell}.$$
Since $\CL_0=\frac{1}{\lambda}\Delta_h^{-1}\CL\Delta_h$, it follows that  \begin{equation}
\label{eq:sup-bound}
    \|\CL_0^ng\|_\infty\leq \frac{1}{\inf h}\|\tfrac{1}{\lambda^n}\CL^n(hg)\|_\infty\leq \frac{A}{\inf h}\nu(|hg|)\beta^{n-\ell}\leq \frac{A\|h\|_\infty}{\inf h}\beta^{n-\ell}.
\end{equation}

To prove the bound on $h$, 
let $Q$ denote the transition matrix of $X^+$ and let $M$  be its primitivity exponent. A (sharp) bound for $M$, due to Wielandt~\cite{wielandt}, is
\begin{equation}
    \label{eq:bound-M}
M\leq (s-1)^2+1.
\end{equation}
Let $u$ be a left Perron-Frobenius eigenvector for $Q$ and define $\psi\colon X^+\to \R$ by 
\begin{equation}
\label{def:psi}    
\psi(x)=u_{x(0)}.
\end{equation}
Then $$(\CL \psi)(x)=\sum_{i}Q_{ix(0)}\psi(ix(0)x(1)\dots)=\sum_{i}u_iQ_{ix(0)}=\lambda u_{x(0)}=\lambda \psi(x),$$
so $\psi$ is an eigenfunction for $\lambda.$ 
Since $\lambda$ is a simple eigenvalue of $\CL$, 
the eigenfunction $h$ is unique up to scaling. The assumption that $\nu(h)=1$ therefore implies that  $h=\tfrac{1}{\nu(\psi)}\psi.$ In particular, 
by the definition~\eqref{def:psi} of $\psi$, the eigenfunction $h$ depends only on the  first entry. 

For $n<\ell$, direct computation shows that $\CL_0^ng(x)=\frac{h(x(0))}{\lambda^n h(x(n))}-\mu(f).$ Recall that $h$ is strictly positive and takes finitely many values, $0\leq\mu(f)\leq 1$, and $\lambda>1$. If $\frac{h(x(0))}{\lambda^n h(x(n))}>\mu(f)$, then $$\left|\frac{h(x(0))}{\lambda^n h(x(n))}-\mu(f)\right|\leq \frac{h(x(0))}{\lambda^n h(x(n))}\leq \frac{h(x(0))}{ h(x(n))}\leq \frac{\|h\|_{\infty}}{\inf h}.$$ Otherwise, $$\left|\frac{h(x(0))}{\lambda^n h(x(n))}-\mu(f)\right|\leq 1\leq \frac{\|h\|_{\infty}}{\inf h}.$$ Without loss of generality, we may assume that  $A\geq 1$, and since $n<\ell$, $$\frac{\|h\|_{\infty}}{\inf h}\leq \frac{A\|h\|_{\infty}}{\inf h}\beta^{n-\ell}.$$ 
Combining this with~\eqref{eq:sup-bound}, we conclude that~\eqref{eq:bound-L} holds for all $n\geq 1$.

Letting $u_p=\min_{i}u_i,$ and $u_q=\max_iu_i$, we have $$\frac{\|h\|_\infty}{\inf h}=\frac{u_q}{u_p}.$$ 
Consider the graph described by $Q$. Since $X$ is transitive, there exists a path from $q$ to $p$ of length $t<s,$ i.e. $(Q^t)_{qp}\geq 1.$ Since $u$ is a left eigenvector, we have $$\lambda^tu_p=(uQ^t)_p=\sum_ju_j(Q^t)_{jp}\geq u_q,$$ which gives \begin{equation}\label{evectorbounds}
 \frac{u_q}{u_p}\leq \lambda^t\leq s^s.
\end{equation}
The first inequality follows.

We are left with proving that the bounds on  $A$ and $\beta$ hold. In~\cite{bowen}, it is shown that $$A=\frac{(\|h\|_\infty+K)\sup_{0<r\leq M}(\frac{1}{\lambda^r}\|\CL^r\|)}{1-\eta}\quad \text{ and }
\quad\beta=(1-\eta)^{\frac{1}{M}},$$
 where $\eta=(4\lambda^M\|h\|_\infty)^{-1}$ and $K=\lambda^Me^2.$ A bound on $M$
 is given in~\eqref{eq:bound-M} and so we must find bounds for
 $\lambda$ and $\|h\|_\infty$ in terms of $s$. 
 
Recall that $h_{\tp}(X^+)=\log \lambda$. It follows that $\lambda$ is the Perron-Frobenius eigenvalue of $Q,$ and since $Q$ is an irreducible $s\times s$ matrix with entries in $\{0,1\},$ we have $1<\lambda\leq s$. 
Since by~\eqref{def:psi} $h$ depends only on the first entry, 
the condition $\nu(h)=1$ can be expressed as
\begin{equation}\label{hnormalisation}
    \sum_{i}\nu([i])h(i)=1.
\end{equation}
We also have that $\sum_i\nu([i])=1$, and so $\|h\|_\infty\geq 1.$  
To derive an upper bound on $h$, let $v$ denote the vector $v_i=\nu([i])$. Then, since $\nu$ is an eigenmeasure, 
$$(Qv)_i=\sum_{j}Q_{ij}v_j=\int_{X^+}Q_{ix(0)}d\nu(x)=\int_{X^+} \CL\one_{[i]}d\nu=\lambda\int_{X^+}\one_{[i]}d\nu=\lambda v_i,$$ 
and it follows that $v$ is a right Perron-Frobenius eigenvector. In particular, $v$ is strictly positive. Combining this with (\ref{hnormalisation}),  we have $$\|h\|_\infty\leq \frac{1}{\min_{i}v_i}.$$
A similar argument as used to derive~\eqref{evectorbounds} gives that $\frac{\max_iv_i}{\min_iv_i}\leq \lambda^s,$ and  $\sum_iv_i=1$ implies that $\max_iv_i\geq \frac{1}{s}.$
Combining these, we have $\|h\|_\infty\leq s\lambda^s\leq s^{s+1}.$ 
It follows that $$\frac{1}{4s^{2s^2}}\leq\frac{1}{4s^{s^2+s+1}} \leq\eta\leq \frac{1}{4}.$$

To complete the proof, we bound $\sup_{0<r\leq M}(\frac{1}{\lambda^r}\|\CL^r\|).$
Let $r\geq 1$ and $\phi\in C(X^+)$. Then $$|\CL^r\phi(x)|\leq \sum_{\sigma^ry=x}|\phi(y)|\leq \|\phi\|_\infty\sum_i(Q^r)_{ix(0)}\leq \|\phi\|_\infty s^r\leq \|\phi\|_\infty s^{s^2}.$$ 
Thus we have that $\|\CL^r\|\leq s^r$ and $\sup_{0<r\leq M}(\frac{1}{\lambda^r}\|\CL^r\|)\leq s^{s^2}.$ It follows that 
\begin{equation*}
    A\leq \frac{4s^{2s^2}(1+e^2)}{3}\leq 15s^{2s^2}\quad \text{ and }\quad\frac{3}{4}\leq \beta\leq \left(1-\frac{1}{4s^{2s^2}}\right)^{\frac{1}{s^2}}, 
\end{equation*}
giving the last two inequalities in the proof. 
\end{proof}

 We use this to complete the proof of Theorem~\ref{th:main}. 
\begin{proof}[Proof of Theorem~\ref{th:main}]
Assume that $(X, \sigma)$ is a mixing shift of finite type with measure of maximal entropy $\mu$, and fix $\varepsilon>0$ and an admissible word $w$ of length $\ell$.  Let $(X^+,\sigma)$ be the one-sided subshift version of $(X,\sigma)$.  We note that $\mu$ defines a measure of maximal entropy on $(X^+,\sigma)$ by declaring all cylinder sets in $(X^+,\sigma)$ to have the same measure that $\mu$ gives to them in $(X,\sigma)$ and note that entropy is preserved via this procedure.  Therefore, we prove this theorem in the setting of $(X^+,\sigma)$ as the conclusion only depends on measures of cylinder sets.  In a slight abuse of notation, we continue to use the notation $(X,\sigma)$ for the system rather than $(X^+,\sigma)$.

We begin by recoding $X$ on words of length $F$. In other words, we create a directed graph whose vertices are the words in $\mathcal{L}_F(X)$ and that has a directed edge from $(a_0,\dots,a_{F-1})$ to $(b_0,\dots,b_{F-1})$ if and only if $a_{i+1}=b_i$ for all $0\leq i<F-1$. Then $\tilde{X}$ is the shift of finite type presented by this directed graph.  We refer to this new shift as $\tilde{X}$ and note that $X$ is topologically conjugate to $\tilde{X}$. Let $\psi\colon X\to \tilde X$ be a conjugacy map and $\tilde\mu$ be the measure on $\tilde X$ corresponding to $\mu$. 

If $\ell \geq F,$ then $\psi[w]\subset \tilde X$ is a cylinder of length $\ell-F+1$ in $\tilde X$, and we denote this cylinder by $[\tilde w].$ Set $f=\one_{[\tilde w]}$, $g=f-\tilde \mu(f),$ and $g_n=\CL_0^ng.$ 
By a result of Kadyrov~\cite[Lemma 3.2]{kadyrov}, we have that $$\left|\tilde\mu'(g_{n+1})-\tilde\mu'( g_{n})\right|\leq \sqrt{2}\|g_n\|_\infty(h_{\tilde\mu}-h_{\tilde\mu'})^{1/2}$$
for any invariant probability measure  $\tilde\mu'$. By Proposition~\ref{prop:bounds}, $$\|g_n\|_\infty\leq \frac{A\|h\|_\infty}{\inf h}\beta^{n-\ell+F-1},$$ and in particular, $$\lim_{n\to \infty} g_n=0=\tilde\mu(g).$$ Let $\mu'$ be any invariant probability measure on $X$ and $\tilde \mu'$ be its pushforward under $\psi$.  Since $\psi$ is a topological conjugacy, $h_\mu=h_{\tilde\mu}$ and $h_{\mu'}=h_{\tilde\mu'}.$ Then, we have
\begin{align*}
    \left|\mu'([w])-\mu([w])\right|&=\left|\tilde \mu'(g)-\tilde\mu(g)\right|\\
    &=\lim_{n\to \infty} \left|\tilde \mu'(g)-\tilde \mu'(g_n)\right| \leq \sum_{n=0}^\infty\left|\tilde \mu'(g_{n+1})-\tilde \mu'( g_{n})\right|\\
    &\leq \frac{\sqrt{2}A\beta^{F-1}\|h\|_\infty}{\beta^\ell(1-\beta)\inf h}(h_\mu-h_{\mu'})^{1/2}\leq \frac{\sqrt{2}A\|h\|_\infty}{\beta^\ell(1-\beta)\inf h}(h_\mu-h_{\mu'})^{1/2},
\end{align*}
which is smaller than $\varepsilon$ whenever 
\begin{equation}
    \label{eq:h-drop}
h_\mu-h_{\mu'}<\frac{\varepsilon^2}{\left(\frac{\sqrt{2}A\|h\|_\infty}{\beta^\ell(1-\beta)\inf h}\right)^2}.
\end{equation}
The three bounds given in Proposition~\ref{prop:bounds} imply that 
$$\left(\frac{4}{3}\right)^{\ell} \xi(s)\geq \frac{\sqrt{2}A\|h\|_\infty}{\beta^\ell(1-\beta)\inf h},$$
where $s$ is the alphabet size for $\tilde X$, which is equal to the number of admissible words of length $F$ in $X.$  One can check that, in fact, $|\mu^{\prime}([w])-\mu([w])|<\frac{\varepsilon}{s}$ when $h_{\mu}-h_{\mu^{\prime}}$ satisfies~\eqref{eq:h-drop}.

If $\ell<F$, the argument is similar. In this case, $\psi[w]=\bigsqcup_{i=1}^r[\tilde \alpha_i],$ where each $\tilde \alpha_i$ is a symbol in the alphabet for $\tilde X$ and $r<s$. 
 We apply this argument, replacing $\tilde w$ by  $\tilde\alpha_i$ (and using the bound $\varepsilon/s$ instead of $\varepsilon$), to see that whenever 
$$ h_{\mu}-h_{\mu'}<\frac{\varepsilon^2}{(\frac{4}{3})^{2\ell} (\xi(s))^2},$$
we have $|\tilde\mu'[\tilde\alpha_i]-\tilde\mu[\tilde\alpha_i]|<\frac{\varepsilon}{s}.$ The triangle inequality then gives 
\begin{equation*}
    |\mu'[w]-\mu[w]|\leq \sum_{i=1}^r|\tilde\mu'[\tilde\alpha_i]-\tilde\mu[\tilde\alpha_i]|<\varepsilon.
 \qedhere
 \end{equation*}
\end{proof}

\subsection{A characteristic measure for entropy stable shifts}  
We show that a sufficiently small entropy drop in the SFT cover a shift suffices for the existence of a characteristic measure. 
\begin{theorem}
\label{th:control-word-stats}
    Assume that $\{X_n\}_{n\in\N}$  is the SFT cover for a topologically mixing system $(X, \sigma)$.  Assume that for each $\varepsilon  > 0$ and integers $\ell,j\geq 1$, there is an integer $m = m(\varepsilon, \ell, j)$ 
    such that  
    \begin{equation}
       \label{eq:entropy-bound-2}
\big|    h_\tp(X_m) - h_\tp (X_{m+j})
\big| < \Xi(\varepsilon, \ell, j)
        \end{equation}
    for the function $\Xi(\varepsilon, \ell, j) = \frac{\varepsilon^2}{(\frac{4}{3})^{2\ell} (\xi(s_m))^2}
    $, where $s_m\leq|\CA(X)|^m$ is the number of admissible words of length $m$ in $X_m$ and $\xi$ is defined in~\eqref{eq:xi}.  Then the system $(X, \sigma)$ supports a characteristic measure. 
\end{theorem}

\begin{definition}
\label{def:entropy-stable}
A topologically transitive subshift 
$(X, \sigma)$ with an SFT cover satisfying the conditions of Theorem~\ref{th:control-word-stats}  is said to be  {\em entropy stable}.
\end{definition}

\begin{proof}
We start by producing a measure $\mu$ on the system $(X, \sigma)$ that is a weak* limit of the measures $(\mu_{m})_{m\in\N}$, 
where $\mu_m$ denotes the measure of maximal entropy on $X_m$. 
The (positive) function $\Xi(\varepsilon, \ell, j)$ is monotonically decreasing in each of the parameters:  as $\varepsilon$ tends to zero,  as the integer $\ell$ tends to infinity, and as $j\to \infty$.  Thus if~\eqref{eq:entropy-bound-2} holds for some $m\in\N$, some $\varepsilon_0 > 0$ and some $\ell_0,j_0\in\N$, then~\eqref{eq:entropy-bound-2} also holds for the same $m$, and for $\varepsilon>0$ and $\ell,j\in\N$, whenever $\varepsilon\geq\varepsilon_0$, $\ell\leq\ell_0$, and $j\leq j_0$.  Given $\varepsilon>0$ and $\ell,j\in\N$, let $m(\varepsilon,\ell,j)$ be the value of $m$ for which~\eqref{eq:entropy-bound-2}  holds.  For notational convenience, we define $m_n:=m(1/n,n,n)$.  

    Let $\phi\in\aut(X, \sigma)$ and assume that $\phi$ has range $R$ (we always assume that the range is symmetric so that $\phi^{-1}$ also has range $R$).  Fix a word $w\in\mathcal{L}(X)$.  Choose $n\in\N$ that is larger than $\max\{|w|,2R+1\}$.  
    By the definition of $m_n$,  we have that
    $$ 
    |h_{\tp}(X_{m_n})-h_{\tp}(X_{m_n+n})|<\Xi(1/n,n,n).
    $$ 
    Since $h_{\mu_{m_n+n}}(\sigma)=h_{\tp}(X_{m_n+n})$, we have that   
\begin{equation}\label{eq:meas1}
    h_{\mu_{m_n+n}}(\sigma)>h_{\tp}(X_{m_n})-\Xi(1/n,n,n).
    \end{equation}
    We also have that $h_{\phi_*(\mu_{m_n+n})}(\sigma)=h_{\mu_{m_n+n}}(\sigma)$ since $\phi$ is a topological conjugacy between $X_{m_n+n}$ and $\phi(X_{m_n+n})$.  Moreover, since $\phi(X_{m_n+n})\subseteq X_{m_n+n-2R}\subseteq X_{m_n}$, we have that $\phi_*\mu_{m_n+n}$ is a measure on $X_{m_n}$.
    Therefore it also follows that 
    \begin{equation}\label{eq:meas2}
    h_{\phi_*(\mu_{m_n+n})}(\sigma)>h_{\tp}(X_{m_n})-\Xi(1/n,n,n). 
    \end{equation}
    Thus it follows from Theorem~\ref{th:main} and~\eqref{eq:meas1} that 
    \[
    |\mu_{m_n+n}([w]) - \mu_{m_n}([w])| < \frac{1}{n} 
    \]
    and from Theorem~\ref{th:main} and~\eqref{eq:meas2} that 
    \[
    |\phi_*\mu_{m_n+n}([w]) - \mu_{m_n}([w])| < \frac{1}{n}. 
    \]
    Combining these, we obtain that     \begin{equation}\label{eq:meas3}
    |\phi_*\mu_{m_n+n}([w]) - \mu_{m_n+n}([w])| < \frac{2}{n}.
    \end{equation}
    Let $\mu$ be a weak* limit of the sequence $(\mu_{m_n+n})_{n\in\N}$.  Since~\eqref{eq:meas3} holds for all sufficiently large $n$, it follows that 
    \begin{equation}\label{eq:meas4}
    |\phi_*\mu([w]) - \mu([w])| = 0.
    \end{equation}
    As~\eqref{eq:meas4} holds for any word $w\in\mathcal{L}(X)$, it follows that $\phi_*\mu=\mu$.  Since this holds for any $\phi\in\Aut(X)$, it follows that $\mu$ is a characteristic measure on $X$.    
\end{proof}

\subsection{Equivalency of entropy stability and language stability}\label{sec:same}
Language stable shifts have the property that for all $j\in\N$, there exists $m\in\N$ such that $X_m=X_{m+j}$.  Thus any language stable shift is also entropy stable.  For a converse statement, we start with a calculation of the change in entropy that happens when a single word is removed from the language. Though such bounds can possibly be derived from existing results~\cite{lind, ramsey}, we provide a self-contained proof for convenience.

\begin{theorem}
\label{th:entropy-drop}
If $X$ is a nontrivial transitive nearest-neighbor shift of finite type with alphabet size $n$ and the shift of finite type $Y$ is obtained by removing a single word $w$ of length $k > 1$ from the language of $X$, then 
\[
h_\tp(X) - h_\tp(Y) > h_\tp(X) e^{-2(3n+4k)h_\tp(X)}.
\]
\end{theorem}

\begin{proof}
Choose $X, h, n, k, w$ as in the statement. Take any letter $a$ such that $\nu([a]) \geq 1/n$ for some ergodic measure of maximal entropy $\nu$ for $Y$.
Since $a \in \CL(Y)$, there exists a word $u$ such that $ua$ has no occurrences of $w$ and $|u| = n + 3k$. 
By transitivity of $X$, there exist words $t, t'$ of minimal length such that $at$ has $w$ as its suffix and $t'ua$ has $w$ as its prefix. By assumption of minimal length, $0 < |t|, |t'| < n + k$, and neither $at$ nor $t'ua$ contain any  other occurrences of $w$. Gluing these words together on the central $w$ yields a word $atva$ which contains only one occurrence of $w$ (as the suffix of $at$). Furthermore, we have that  
$|v| = |u| + |t'| - k \in (n + 2k, 2n + 3k)$. 

For any $m$, define $Z_m$ to be the set of words in $\CL_m(Y)$ with more than $\frac{m}{2n}$ occurrences of $a$.
By the ergodic theorem, $\nu(Z_m)$ approaches $1$, and so by~\cite[Corollary 2.7]{ramos-pavlov},
\[
m^{-1} \ln |Z_m| \rightarrow h(\nu) = h_\tp(Y)
\]
as $m\to\infty$. 

Choose any $0 < \varepsilon < \frac{1}{2n}$. For any word $y \in Z_m$, define $A_Y = \{i \ : \ y(i) = a\}$. By definition, $|A_Y| > \frac{m}{2n}$, so define $B_Y$ to consist of the smallest $\lfloor \frac{m}{2n} \rfloor$ elements of $A_Y$. Then, for every set $S \subset B_Y$ with $|S| = \lfloor \varepsilon m \rfloor$, associate a word $f(y,S)$ by replacing $y(s) = a$ by $atva$ for each $s \in S$. Then $f(y,S) \in L(X)$, since all adjacencies are either part of $atva$ or $y$, both of which are in $\CL(X)$.

We claim that all such words are distinct. To see this, consider any unequal pairs $(y, S)$ and $(y', S')$. Take the minimal $i$ for which $y(i) \neq y'(i)$ or $\chi_S(i) \neq \chi_{S'}(i)$. Since $y(j) = y'(j)$ and $\chi_S(j) = \chi_{S'}(j)$ for $j < i$, the replacements made in $y$ and $y'$ to the left of the $i$th location are identical, resulting in a word $z$. Then the letters immediately after $z$ in $f(y,S)$ and $f(y',S')$ are $y(i)$ and $y'(i)$ respectively, and so if $y(i) \neq y'(i)$ then $f(y,S) \neq f(y',S')$. 

Thus we can assume $y(i) = y'(i)$ and $\chi_S(i) \neq \chi_{S'}(i)$, and without loss, we assume  that $i \in S$ and $i \notin S'$, meaning that $y(i) = y'(i) = a$. Since $i \in S$, $f(y,S)$ begins with $zatva$.  Define $j = \min\{s \in S' \ : \ s > i\}$. If 
$j > i + |t|$, then the $|t| + 1$ letters after $y'(i)$ are unchanged in the creation of $f(y',S')$, meaning that $f(y',S')$ begins with $zy'(i+1) \ldots y'(i + |t| + 1)$. Since $y'$ contained no $w$ and $at$ ends with $w$, $zat \neq zy'(i+1) \ldots y'(i + |t| + 1)$, and so again $f(y, S) \neq f(y', S')$.

The only remaining case is $i < j \leq i + |t|$. This means that $f(y',S')$ begins with $zaratva$ for some word $r$ with 
$0 \leq |r| \leq |t|$. Then the occurrence of $w$ at the end of $at$ ends with $f(y', S')(p)$ for some $p \in [|z| + 2 + |t|, |z| + 1 + 2|t|]$. If $f(y, S) = f(y', S')$, then $f(y, S)$ would have an occurrence of $w$ at the same location. However, 
\[
|z| + 1 + |t| < p \leq |z| + 2 + 2|t| < |z| + |t| + n + k < |z| + |t| + |v|,
\] 
and so $f(y,S)(p)$ lies within the $va$ following $zat$ at the beginning of $f(y, S)$. But $atva$ contains only one occurrence of $w$, ending at the final letter of $t$, a contradiction. Therefore, in all cases, $f(y, S) \neq f(y', S')$.

Every $f(y, S)$ has length $m + \lfloor \varepsilon m \rfloor (1 + |t| + |v|) < m + m\varepsilon(3n+4k)$. Putting this together yields
\[
|L_{m(1 + \varepsilon(3n+4k))}(X)| \geq |Z_m| \cdot \binom{\lfloor \frac{m}{2n} \rfloor}{\lfloor \varepsilon m \rfloor}.
\]
Taking logarithms dividing by $m$, and letting $m$ tend to infinity, to get
\[
h_\tp(X) (1 + \varepsilon(3n+4k)) \geq h_\tp(Y) - \varepsilon \ln(2 \varepsilon n).
\]
Choose $\varepsilon = (2n)^{-1} e^{-2(3n+4k)h_\tp(X)} < \frac{1}{2n}$. Then $\ln(2 \varepsilon n) = -2(3n+4k)h_\tp(X)$, yielding
\[
h_\tp(X) + \varepsilon(3n+4k) h_\tp(X) \geq h_\tp(Y) + 2\varepsilon (3n+4k)h_\tp(X).\]
Thus it follows that 
\[ h_\tp(X) - h_\tp(Y) \geq \varepsilon(3n+4k) h_\tp(X).
\]
Finally,
\begin{multline*}
h_\tp(X) - h_\tp(Y) \geq \varepsilon(3n+4k) h_\tp(X) \\ = h_\tp(X) (3n+4k) (2n)^{-1} e^{-2(3n+4k)h_\tp(X)} > h_\tp(X) e^{-2(3n+4k)h_\tp(X)}. \quad\qedhere
\end{multline*}

\end{proof}

\begin{corollary}
If $X$ is an infinite transitive subshift with alphabet of size $a$ and SFT cover
$\{X_n\}_{n\in\N}$, then for any $n$ and $k$ for which $X_n \supsetneq X_{n+k-1}$, we have that 
\[
h_\tp(X_n) - h_\tp(X_{n+k-1}) > (\ln 2) s^{-1} e^{-3(3s+4k)\ln a}, 
\]
where  $s = |\CL_n(X_n)|$. 
\end{corollary}

\begin{proof}
Take any such $X$, $n$, $k$. Since $X$ is transitive, each $X_n$ is also transitive. Also, there exists $w \in \CL_{n+k-1}(X_n) \setminus 
\CL_{n+k-1}(X_{n+k-1})$. By recoding, we may assume that $X_n$ is a nearest-neighbor SFT and that $w$ is of length $k$.

As a nearest-neighbor SFT, $X_n$ has alphabet $\CL_n(X_n)$ of size $s$, and $X_{n+k-1}$ is contained in the SFT $Y$ obtained by removing the $k$-letter word $w$ from $\CL_k(X)$. 
Therefore by Theorem~\ref{th:entropy-drop},
\begin{multline*}
h_\tp(X_n) - h_\tp(X_{n+k-1}) \geq h_\tp(X_n) - h_\tp(Y) \\
> h_\tp(X_n) e^{-2(3s+4k)h_\tp(X_n)}. 
\end{multline*}

It is clear that $h_\tp(X_n) \leq \ln a$. Finally, since $X_n$ is a transitive nearest-neighbor SFT  with $s$ letters, it has a cycle $C$ of minimum length, which must have length at most $s$, beginning and ending at some letter $L$. Since $X_n$ is infinite, there must be another cycle $C'$ beginning and ending at $L$, which we can assume to not repeat a letter, and so which also has length at most $s$. All concatenations of $C$ and $C'$ yield points of $X_n$, and so $h_\tp(X_n) \geq \frac{\ln 2}{s}$. Therefore,
\[
 h_\tp(X_n) e^{-2(3s+4k)h_\tp(X_n)} \geq (\ln 2) s^{-1} e^{-3(3s+4k)\ln a},
\]
completing the proof.
\end{proof}

It follows immediately from this corollary that the classes of language stable and entropy stable are the same.
\begin{theorem}
A transitive subshift is language stable if and only if it is entropy stable. 
\end{theorem}

\section{Periodic point stability}

\label{sec:periodic}

\subsection{Period stable shifts and characteristic measures}
For a shift $(X, \sigma)$ and $p\in\N$, let $\CP_p = \CP_p(X)$ denote the set of periodic points of minimal period $p$.  

\begin{definition}\label{def:periodstable}
The shift $(X, \sigma)$ with SFT cover $\{X_n\}_{n\in\N}$ is {\em period stable} if for all $m\in\N$, there exist $n, p\in \N$ such that $\CP_p(X_n) = \CP_p(X_{n+m})$. 
\end{definition}

Note that the assumption that this holds for all $m\in\N$ is equivalent to the (seemingly weaker) assumption that this condition holds for infinitely many $m\in\N$.  We note that the property of being period stable is preserved under topological conjugacies.

\begin{lemma}
If  $X\subseteq\CA^{\Z}$ is a subshift that is period stable and  $Y\subseteq\mathcal{B}^{\Z}$ is a subshift that is topologically conjugate to $X$, then $Y$ is period stable.
\end{lemma}
\begin{proof}
Let $\varphi\colon X\to Y$ be a topological conjugacy between $X$ and $Y$.  Let $R$ be a range that is common to both $\varphi$ and $\varphi^{-1}$.  For words $w\in\mathcal{L}_n(X)$ with $n>2R$, we abuse notation by writing $\varphi(w)$ for the word of length $n-2R$ obtained by applying the range $R$ block map implementing $\varphi$ to $w$, and similarly we do this for $\varphi^{-1}(w)$ when $w\in\mathcal{L}_n(Y)$.  We also extend the domain of $\varphi$ from $X$ to $X_n$, for all $n>2R$, by defining it to be the function determined by this range $R$ block code on the larger domain.  Similarly for $\varphi^{-1}$ and $Y_n$.  Note that $\varphi^{-1}\circ\varphi$ is defined on $X_n$ only when $n>4R$ but, when it is defined, it is the identity map (defined by the identity block code of range $2R$).  This implies, when $n>4R$, that the shifts $X_n$ and $\varphi(X_n)$ are topologically conjugate and $\varphi$ is a conjugacy.  In particular, when $n>4R$, we have $\varphi(\CP_p(X_n))=\CP_p(\varphi(X_n))$ for any $p$.

Let $\{X_n\}_{n=1}^{\infty}$ be the SFT cover of $X$ and let $\{Y_n\}_{n=1}^{\infty}$ be the SFT cover of $Y$.  Fix $m\in\N$ and choose $n,p\in\N$ such that $\CP_p(X_n)=\CP_p(X_{n+(4R+m)})$.  Without loss of generality, we can assume that $n>4R$ (otherwise we can take $n^{\prime},p^{\prime}$ such that $\CP_{p^{\prime}}(X_{n^{\prime}})=\CP_{p^{\prime}}(X_{n^{\prime}+8R+m})$ and set $n=n^{\prime}+4R$ and $p=p^{\prime}$).  Notice that for any $w\in\mathcal{L}_{n+4R}(X)$, we have $\varphi(w)\in\mathcal{L}_{n+2R}(Y)$.  Therefore $\varphi(X_{n+4R})\subseteq Y_{n+2R}$.  Similarly, $\varphi^{-1}(Y_{n+2R})\subseteq X_n$ and so $\varphi(\varphi^{-1}(Y_{n+2R}))\subseteq\varphi(X_n)$.  Since $n>2R$, $\varphi(\varphi^{-1}(Y_{n+2R}))=Y_{n+2R}$ and so $Y_{n+2R}\subseteq\varphi(X_n)$.  In other words, 
$$ 
\varphi(X_{n+4R})\subseteq Y_{n+2R}\subseteq\varphi(X_n). 
$$ 
Similarly, we have 
$$ 
\varphi(X_{n+4R+m})\subseteq Y_{n+2R+m}\subseteq\varphi(X_{n+m}). 
$$ 
Since $\CP_p(X_n)=\CP_p(X_{n+4R+m})$, we have that $\CP_p(X_n)=\CP_p(X_{n+k})$ for all $0\leq k\leq4R+m$.  In particular, 
$$
\CP_p(X_n)=\CP_p(X_{n+4R})=\CP_p(X_{n+m})=\CP_p(X_{n+4R+m}). 
$$
Since $\CP_p(\varphi(X_{n+4R}))=\varphi(\CP_p(X_{n+4R}))=\varphi(\CP_p(X_n))=\CP_p(\varphi(X_n))$, we have that 
$$
\CP_p(Y_{n+2R})=\varphi(\CP_p(X_n)). 
$$
Similarly, we have 
$$
\CP_p(Y_{n+2R+m})=\varphi(\CP_p(X_{n+m})). 
$$
So $\CP_p(Y_{n+2R})=\CP_p(Y_{n+2R+m})$.  As this holds for any $m\geq 1$, the subshift $Y$ is period stable.
\end{proof}

We next check that this condition suffices for producing a characteristic measure. 
\begin{theorem}
Every period stable subshift has a characteristic measure. 
\end{theorem}
\begin{proof}
Assume that $(X, \sigma)$ is a period stable shift and let $\{X_n\}_{n\in\N}$ be its SFT cover.  
For each $m\in\N$, choose $n_m, p_m\in\N$ such that $\CP_{p_m}(X_{n_m})=\CP_{p_m}(X_{n_m+m})$.  For ease of notation, we write $\CP_m:=\CP_{p_m}(X_{n_m})=\CP_{p_m}(X_{n_m+m})$ and define 
$$ 
\mu_m:=\frac{1}{|\mathcal{P}_m|}\cdot\sum_{x\in\mathcal{P}_m}\delta_x, 
$$ 
where $\delta_x$ denotes the Dirac measure at $x$.  Then the measure $\mu_m$ is an invariant measure supported on $X_{n_m+m}$, which  is a subshift of $X_{n_m}$.  Let $\mu$ be a weak* limit point of the sequence $\{\mu_m\}_{m=1}^{\infty}$.  We claim that $\mu$ is a characteristic measure supported on $X$.

Note that for any $w\notin\mathcal{L}(X)$ and any $m>|w|$, we have $\mu_m([w])=0$,  and therefore we also have that $\mu([w])=0$.  It follows that  the support of $\mu$ does not intersect $[w]$ for any forbidden $w\in\mathcal{F}(X)$.  In particular, the support of $\mu$ is contained in $X$.

Let $\varphi\in\Aut(X)$ be given and choose $R\in\N$ such that $\varphi\in\Aut_R(X)$.  Let $\Phi$ and $\Phi^{-1}$ be range $R$ block codes that implement $\varphi$ and $\varphi^{-1}$, respectively.  
For any $m\geq4R$, note that $n_m+m\geq4R+1$ and so $\Phi^{-1}\circ\Phi$ implements the identity map on $X_{n_m+m}$ as a range $2R$ block code.  Therefore $$(\varphi^{-1}\circ\varphi)\colon X_{n_m+m}\to X_{n_m+m-4R}$$ is the identity map.  It follows that $X_{n_m+m}$ is topologically conjugate to $\varphi(X_{n_m+m})$ with the conjugacies implemented by $\varphi$ and $\varphi^{-1}$. 
In general $X_{n_m+m}$ and $\varphi(X_{n_m+m})$ are not equal.  However, both $X_{n_m+m}$ and $\varphi(X_{n_m+m})$ are subshifts of $X_{n_m}$, since $\varphi(X_{n_m+m})\subseteq X_{n_m+m-2R}\subseteq X_{n_m}$ (we have used the fact that $m>2R$).  Therefore, $\CP_{p_m}(\varphi(X_{n_m+m}))\subseteq\CP_{p_m}(X_{n_m})$.  Since $X_{n_m+m}$ and $\varphi(X_{n_m+m})$ are topologically conjugate, we have $|\CP_{p_m}(X_{n_m+m})|=|\CP_{p_m}(\varphi(X_{n_m+m}))|$.  It therefore follows from our assumption, that $\CP_{p_m}(X_{n_m})=\CP_{p_m}(X_{n_m+m})$, that $\CP_{p_m}(\varphi(X_{n_m+m}))=\CP_{p_m}(X_{n_m+m})$.  In particular, $\varphi$ acts like a permutation on $\mathcal{P}_m$ and so $\varphi_*\mu_m=\mu_m$.  Since this holds for any $m\geq4R+1$, it follows that  $\varphi_*\mu=\mu$.  Since $\varphi\in\Aut(X)$ is arbitrary, this holds for all elements of $\Aut(X)$ and thus $\mu$ is a characteristic measure.
\end{proof} 

We devote the remainder of this section to an example showing that this theorem covers more than the subshifts that are period stable for simple reasons, such as those with a periodic point or those that are language stable.  
{\em A priori}, we do not know if the subshift we construct has non-trivial automorphisms and a non-trivial automorphism group, and so in Section~\ref{sec:nontrivial} we indicate how to obtain such a system. 

\subsection{Construction of a period stable subshift with no periodic points that is not language stable}
\label{sec:construction}

Let $Y\subseteq\{0,1\}^{\Z}$ be a minimal Sturmian shift satisfying the following two conditions:
    \begin{itemize}
        \item $11$ is forbidden in $Y$; 
        \item $000$ is forbidden in $Y$.
    \end{itemize}
Sturmians satisfying these two conditions are easy to construct, for instance by taking the cutting sequence~\cite[Chapter 6]{Pyth} of a line with irrational slope $\alpha\in(1/2,1)$.  Let $y\in Y$ be a fixed element of $Y$ with $y(1)=0$.  We use $y$ to construct a subshift $X\subseteq\{0,1\}^{\Z}$ that is period stable but not language stable and has no periodic points.  To define $X$, we specify its set $\mathcal{F}$ of forbidden words, setting $\mathcal{F}$ to be the union of four sets of words, 
$$
\mathcal{F}=\mathcal{W}_1\cup\mathcal{W}_2^{11}\cup\mathcal{W}_2^{00}\cup\mathcal{W}_2^{000}, 
$$
and we specify these words in the construction. 

We refer to the elements of $\mathcal{W}_1$ as the {\em Type I} forbidden words and to the elements of $\mathcal{W}_2^{11}\cup\mathcal{W}_2^{00}\cup\mathcal{W}_2^{000}$ as the {\em Type II} forbidden words, reflecting the  different roles the words play in our construction.  Roughly speaking, the Type I words ensure $X$ is not language stable and the Type II words ensure $X$ does not contain any periodic points.

\subsubsection{Some auxiliary results}
Before constructing the forbidden words $\CF$ that define the subshift, we introduce some lemmas that are used to control the entropy.  For a topologically transitive shift $Z$ of finite type and for a word $w\in\CL(Z)$, let $Z(w)$ denote the subshift obtained by setting 
$$ 
Z(w):=\{z\in Z\colon\sigma^i(z)\notin[w]\text{ for all }i\in\Z\}. 
$$
The first result we use is due to Lind.
\begin{lemma}[{Lind~\cite[Theorem 3]{lind}}]\label{lem:lind1}
Let $Z$ be a topologically transitive subshift of finite type and let $\varepsilon>0$ be given.  There exists $k\in\N$ such that for any word $w\in\mathcal{L}(Z)$ with $|w|\geq k$, we have $|h_{\tp}(Z)-h_{\tp}(Z(w))|<\varepsilon$.
\end{lemma}

Lind's result holds in greater generality but we could not find a specific reference in the literature.  So, we state and prove the generalization for subshifts that are not necessarily topologically transitive.

\begin{lemma}\label{lem:lind}
Let $Y$ be a subshift of finite type and let $\varepsilon>0$ be given.  There exists $k\in\N$ such that for any word $w\in\mathcal{L}(Y)$ with $|w|\geq k$, we have $|h_{\tp}(Y)-h_{\tp}(Y(w))|<\varepsilon$.
\end{lemma}
\begin{proof} 
Since $Y$ is a subshift of finite type, there is a topologically transitive subshift of finite type, $Z$, contained in $Y$, that satisfies $h_{\tp}(Z)=h_{\tp}(Y)$. (To see this, one can apply, for example, \cite[Lemma 4.7]{CKP} in the special case that $a$ is the size of the language of $Y$, $f$ is the length of the longest forbidden word among a finite set of forbidden words used to define $Y$, $X:=Y$, and $\varphi$ is the identity map.)  By Lemma~\ref{lem:lind1}, we can find $k\in\N$ such that for any word $w\in\mathcal{L}(Z)$ with $|w|\geq k$, we have $|h_{\tp}(Z)-h_{\tp}(Z(w))|<\varepsilon$ for any $w\in\mathcal{L}(Z)$ with $|w|\geq k$.  
By construction of $Z$, this implies that $|h_{\tp}(Y)-h_{\tp}(Z(w))|<\varepsilon$.  Since $Z(w)\subseteq Y(w)$, it follows that  $|h_{\tp}(Y)-h_{\tp}(Z(w))|<\varepsilon$.  Finally note that if $w\in\mathcal{L}(Y)\setminus\mathcal{L}(Z)$ then $Z=Z(w)$ and since $Z(w)\subseteq Y(w)$ we have $|h_{\tp}(Y)-h_{\tp}(Z(w))|=0$.
\end{proof}

\subsubsection{The set of Type I forbidden words $\mathcal{W}_1$} 
\label{subsec:W1}
For each integer $n>0$, define 
$$
u_n:=111y(1)y(2)y(3)\dots y(4n)111
$$
where $y(i)$ is the $i^{th}$ coordinate of the element $y\in Y$ that we fixed at the beginning of this construction.  We index the letters of $u_n$ starting from $1$, writing 
$$ 
u_n(i)=\left\{\begin{tabular}{cl} 
$1$ & if $1\leq i\leq3$; \\ 
$y(i-3)$ & if $4\leq i\leq4n+3$; \\ 
$1$ & if $4n+4\leq i\leq4n+6$. 
\end{tabular}\right.
$$ 
We set  
$$
\mathcal{W}_1=\{u_n\colon n\in\N\}.
$$
Note that no two elements of $\mathcal{W}_1$ have the same lengths and the lengths of the elements of $\mathcal{W}_1$ are the numbers of the form $4n+6$.  In particular, the lengths form a syndetic subset of $\N$.  Further note that the word $y(1)y(2)y(3)\dots y(4n)$ does not have $11$ as a subword, because $11$ is a forbidden word in the shift $Y$. Therefore, the only places that $11$ occurs as a subword of $u_n$ are:
    \begin{enumerate}
        \item within its prefix $111$ (the word $u_n(1)u_n(2)u_n(3)$); \label{pos1}
        \item within its suffix $111$ (the word $u_n(4n+4)u_n(4n+5)u_n(4n+6)$); \label{pos2}
        \item if $u_n(4n+3)=1$, then within its suffix $1111$ (the word $u_n(4n+3)u_n(4n+4)u_n(4n+5)u_n(4n+6)$). \label{pos3}
    \end{enumerate}
Recall that we chose $y\in Y$ such that $y(1)=0$, and it follows that $u_n$ begins with the prefix $1110$ for all $n$.  It ends with the suffix either $0111$ or $01111$ depending on whether $u_n(4n+3)$ is $0$ or $1$.

\subsubsection{The set of Type II words in $\mathcal{W}_2^{11}$}\label{sec:woneone}
Consider the shift 
$$
\mathcal{X}_1=\{x\in\{0,1\}^{\mathbb{Z}}\colon\sigma^i(x)\notin[u_n]\text{ for all }i\in\Z\text{ and }n\in\N\}. 
$$
Let $\mathcal{H}_1$ be the subshift of finite type whose only forbidden word is $111$.  Note that if $h\in\mathcal{H}_1$, then $h$ does not contain $111$ as a subword.  
In particular, since $u_n$ itself contains $111$ as a subword, $h$ does not contain $u_n$ as a subword for any $n\in\N$.  It follows that $\mathcal{H}_1\subseteq\mathcal{X}_1$ and so 
$$ 
h_{\tp}(\mathcal{X}_1)\geq h_{\tp}(\mathcal{H}_1)>\log_2(1.839).
$$ 
Hence the topological entropy of $\mathcal{X}_1$ is positive.  
Furthermore, $\mathcal{X}_1$ is topologically mixing: if $w_1,w_2\in\mathcal{L}(\mathcal{X}_1)$ then $^{\infty}0w_10^kw_20^{\infty}\in\mathcal{X}_1$ for any $k\geq3$ (since $u_n$ does not contain $000$ as a subword for any $n$).

We inductively construct the set $\mathcal{W}_2^{11}$.  Our goal is to forbid additional words from the language of $\mathcal{X}_1$ that eliminate all periodic points that contain $11$ as a subword, while ensuring that the resulting shift has entropy close to that of $\mathcal{X}_1$ and such that the newly forbidden words do not occur as subwords of $y$.

Let $\mathcal{P}_0\subseteq\mathcal{X}_1$ be the collection of all periodic points that contain $11$ as a subword.  Among the elements of $\mathcal{P}_0$,  choose a periodic point $p_0$ that has the least possible minimal period.  Find a word $v_0\in\mathcal{L}(\mathcal{X}_1)$ such that $v_0$ has $11$ as a prefix and such that $p_0=\dots v_0v_0v_0v_0v_0\dots$ is (a shift of) the bi-infinite self-concatenation of $v_0$.  Next we choose an integer $k_0>0$ sufficiently large such that $v_0^{k_0}$ does not occur as a subword of $y$; we note we can always find such a word, as $y\in Y$ can not contain arbitrarily long subwords that are periodic with a fixed period because $Y$ is  a Sturmian shift.  
Define the word $z_0:=v_0^{100k_0}11$.  Since $v_0$ begins with prefix $11$, the word $z_0$ is periodic with period $|v_0|$ and has $11$ as both a prefix and a suffix.  Define 
$$ 
\mathcal{P}_1:=\{p\in\mathcal{P}_0\colon z_0\text{ is not a subword of }p\}. 
$$ 
We continue this procedure inductively.  Assume that we have constructed words $z_0,\dots,z_r$ such that: 
\begin{enumerate}
    \item each $z_i$ is periodic; 
    \item the lengths of the $z_i$ are increasing ; 
    \item each $z_i$ has $11$ as both a prefix and a suffix; 
    \item there is a word, $v_i$, that has $11$ as a prefix, and an integer $k_i$ such that $z_i=v_i^{100k_i}11$ where $k_i$ is sufficiently  large that $v_i^{k_i}$ does not occur as a subword of $y$. 
\end{enumerate}   Suppose further that we have defined $\mathcal{P}_0\supset\mathcal{P}_1\supset\dots\supset\mathcal{P}_r\supset\mathcal{P}_{r+1}$ such that for each $i=0,1,\dots,r$, 
$$ 
\mathcal{P}_{i+1}=\{p\in\mathcal{P}_i\colon z_i\text{ is not a subword of }p\} 
$$ 
and the word $\dots z_iz_iz_iz_iz_i\dots$ is an element of $\mathcal{P}_i$ of minimum possible period, among all elements of $\mathcal{P}_i$.  Among the elements of $\mathcal{P}_{r+1}$ choose a periodic point $p_{r+1}$ that has the smallest possible minimal period.  Find a word $v_{r+1}\in\mathcal{L}(\mathcal{X}_1)$ such that $v_{r+1}$ has $11$ as a prefix and such that $p_{r+1}=\dots v_{r+1}v_{r+1}v_{r+1}v_{r+1}v_{r+1}\dots$ is (a shift of) the bi-infinite self-concatenation of $v_{r+1}$.  Choose an integer $k_{r+1}>0$ sufficiently large such  that $v_{r+1}^{k_{r+1}}$ does not occur as a subword of $y$ and is also such that $|v_{r+1}|\cdot k_{r+1}>|z_r|$.  Define $z_{r+1}:=v_{r+1}^{100k_{r+1}}11$.  Since $v_{r+1}$ begins with prefix $11$, the word $z_{r+1}$ is periodic and has $11$ as both a prefix and a suffix.  Define 
$$ 
\mathcal{P}_{r+2}:=\{p\in\mathcal{P}_{r+1}\colon z_{r+1}\text{ is not a subword of }p\}. 
$$ 
Inductively, this defines $\mathcal{P}_i$ for all $i\geq0$.  We define 
\begin{equation}\label{eq:words_two_one-one}
\mathcal{W}_2^{11}:=\{z_i\colon i\in\N\} 
\end{equation} 

\subsubsection{The set of Type II words in $\mathcal{W}_2^{00}$}\label{sec:zerozero}
Consider the shift 
$$
\mathcal{X}_2=\{x\in\mathcal{X}_1\colon\sigma^i(x)\notin[z]\text{ for all }i\in\Z\text{ and }z\in\mathcal{W}_2^{11}\}. 
$$
Let $\mathcal{H}_2$ be the subshift of finite type whose only forbidden word is $11$.  Note that $\mathcal{H}_2\subset\mathcal{H}_1$ and that $\mathcal{H}_2\subseteq\mathcal{X}_2$ because every word in $\mathcal{W}_2^{11}$ contains $11$ as a subword.  Consequently, 
$$ 
h_{\tp}(\mathcal{X}_2)\geq h_{\tp}(\mathcal{H}_2)>\log_2(1.618). 
$$ 
Hence the topological entropy of $\mathcal{X}_2$ is positive.  
Moreover, $\mathcal{X}_2$ is topologically mixing: if $w_1,w_2\in\mathcal{L}(\mathcal{X}_2)$, then $^{\infty}0w_10^kw_20^{\infty}\in\mathcal{X}_2$ as long as $k\geq\max\{|w_1|,|w_2|,3\}$ 
(analogous to the argument for mixing on words in $\mathcal W_2^{11}$,  
the shift $\mathcal X_2$ contains no  element of $\mathcal{W}_1$ and it contains no element of $\mathcal{W}_2^{11}$ as a subword because all such words begin and end with $11$ and contain at least $100$ copies of their period).

We  inductively construct the set $\mathcal{W}_2^{00}$.  Our goal is to forbid additional words from the language of $\mathcal{X}_2$ that  eliminate all periodic points that contain $00$ but do not contain $000$ as a subword.

Let $\mathcal{Q}_0\subseteq\mathcal{X}_2$ be the collection of all periodic points that contain $00$ as a subword but do not contain $000$ as a subword.  These periodic points do not contain $11$ as a subword because $\mathcal{X}_2\subseteq\mathcal{X}_1$ was constructed to remove all periodic points from $\mathcal{X}_1$ that contain $11$ as a subword.  We  proceed as in the construction for words in $\mathcal W_2^{11}$.  
Among the elements of $\mathcal{Q}_0$, choose a periodic point $q_0$ that has the least  possible minimal period.  Find a word $w_0\in\mathcal{L}(\mathcal{X}_2)$ such that $w_0$ has $00$ as a prefix and such that $q_0=\dots w_0w_0w_0w_0w_0\dots$ is (a shift of) the bi-infinite concatenation of $w_0$.  Choose an integer $\ell_0>0$ sufficiently large such that $w_0^{\ell_0}$ does not occur as a subword of $y$.  Define the word $\alpha_0:=w_0^{100\ell_0}00$.  Since $w_0$ begins with prefix $00$, the word $\alpha_0$ is periodic and has $00$ as both a prefix and a suffix.  Define 
$$ 
\mathcal{Q}_1:=\{q\in\mathcal{Q}_0\colon\alpha_0\text{ is not a subword of }q\}. 
$$ 
We continue this procedure inductively.  
Assume we have constructed words $\alpha_0,\dots,\alpha_r$ such that 
\begin{enumerate}
    \item 
each $\alpha_i$ is periodic; 
the lengths of the $\alpha_i$ are increasing; 
\item each $\alpha_i$ has $00$ as both a prefix and a suffix; 
\item there is a word $w_i$ and an integer $\ell_i$ such that $\alpha_i=w_i^{100\ell_i}00$ and $w_i^{\ell_i}$ is not a subword of $y$. 
\end{enumerate}
We further suppose  that we have defined $\mathcal{Q}_0\supset\mathcal{Q}_1\supset\dots\supset\mathcal{Q}_r\supset\mathcal{Q}_{r+1}$ such that for each $i=0,1,\dots,r$, 
$$ 
\mathcal{Q}_{i+1}=\{q\in\mathcal{Q}_i\colon \alpha_i\text{ is not a subword of }q\} 
$$ 
and the word $\dots \alpha_i\alpha_i\alpha_i\alpha_i\alpha_i\dots$ is an element of $\mathcal{Q}_i$ of minimum possible period, among all elements of $\mathcal{Q}_i$.  
Among the elements of $\mathcal{Q}_{r+1}$, we  choose a periodic point $q_{r+1}$ that has the smallest possible minimal period.  Find a word $w_{r+1}\in\mathcal{L}(\mathcal{X}_2)$ such that $w_{r+1}$ has $00$ as a prefix and such that $q_{r+1}=\dots w_{r+1}w_{r+1}w_{r+1}w_{r+1}w_{r+1}\dots$ is (a shift of) the bi-infinite self-concatenation of $w_{r+1}$.  Choose an integer $\ell_{r+1}>0$ sufficiently large that $w_{r+1}^{\ell_{r+1}}$ does not occur as a subword of $y$ and is also such that $|w_{r+1}|\cdot \ell_{r+1}>|\alpha_r|$.  Define $\alpha_{r+1}:=w_{r+1}^{100\ell_{r+1}}00$.  Since $w_{r+1}$ begins with prefix $00$, the word $\alpha_{r+1}$ is periodic and has $00$ as both a prefix and a suffix.  Define 
$$ 
\mathcal{Q}_{r+2}:=\{q\in\mathcal{Q}_{r+1}\colon \alpha_{r+1}\text{ is not a subword of }q\}. 
$$ 
Inductively, this defines $\mathcal{Q}_i$ for all $i\geq0$.  We define 
\begin{equation}\label{eq:words_two_zero-zero-zero}
\mathcal{W}_2^{00}:=\{\alpha_i\colon i\in\N\}. 
\end{equation} 

\subsubsection{The set of Type II words in $\mathcal{W}_2^{000}$}\label{sec:zerozerozero}
Consider the shift 
$$
\mathcal{X}_3=\{x\in\mathcal{X}_2\colon\sigma^i(x)\notin[\alpha]\text{ for all }i\in\Z\text{ and }\alpha\in\mathcal{W}_2^{00}\}. 
$$
Let $\mathcal{H}_3$ be the subshift of finite type whose only forbidden words are $11$ and $1001$.  Note that if $h\in\mathcal{H}_3$ then $h$ contains no  word in $\mathcal{W}_1\cup\mathcal{W}_2^{11}$ as a subword (since all such words contain $11$ as a subword and $h$ does not) and also  contains no word in $\mathcal{W}_2^{00}$ as a subword (since all such words contain $1001$ as a subword and $h$ does not).  Therefore $\mathcal{H}_3\subseteq\mathcal{X}_3$ and  
$$ 
h_{\tp}(\mathcal{X}_3)\geq h_{\tp}(\mathcal{H}_3)>\log_2(1.512).
$$ 
In particular, the topological entropy of $\mathcal{X}_3$ is positive.  Moreover $\mathcal{X}_3$ is topologically mixing: for any $w_1,w_2\in\mathcal{L}(\mathcal{X}_3)$ then $^{\infty}0w_10^kw_20^{\infty}\in\mathcal{X}_2$ so long as $k\geq\max\{|w_1|,|w_2|,3\}$, and it is also easy to  check that $\mathcal{H}_3$ is also topologically mixing.

Since $\mathcal{X}_3$ contains a positive entropy subshift of finite type, it has infinitely many periodic points.  By construction of the sets $\mathcal{W}_2^{11}$ and $\mathcal{W}_2^{00}$, none of these periodic points contains $11$ as a subword and the only periodic point that does not contain $000$ as a subword is the point $(01)^\infty$. 
As before, we  inductively construct the set $\mathcal{W}_2^{000}$ with the goal of eliminating all of the remaining periodic points.  However, the words in $\mathcal{W}_2^{000}$  have to be chosen more delicately than before to ensure that we obtain a nonempty period stable but not language stable shift.  Setting  $\mathcal{W}_2^{000}=\{\beta_0,\beta_1,\dots\}$,  our goal is to construct the words in $\mathcal{W}_2^{000}$ such that for any fixed $n$, the shift 
$$ 
\mathcal{X}_3(n):=\{x\in\mathcal{X}_3\colon\sigma^{i}(x)\notin[\beta_j]\text{ for all }i\in\Z\text{ and }0\leq j\leq n\} 
$$ 
has entropy bounded below by $\frac{1}{2}h_{\tp}(\mathcal{H}_3)$, and 
such that the words are chosen such that the shift
$$ 
\mathcal{X}_4:=\bigcap_{n=0}^{\infty}\mathcal{X}_3(n) 
$$ 
is period stable but not language stable.  Largely this is  accomplished by defining the words $\beta_i$ to be long periodic words, with carefully chosen lengths, that eliminate the remaining periodic points. 

Again we define words $\beta_0,\beta_1,\dots$ inductively, and for each $t\geq 0$ we then set 
$$ 
\mathcal{H}_3(t)=\{x\in\mathcal{H}_3\colon\sigma^i(x)\notin[\beta_j]\text{ for all }i\in\Z\text{ and }0\leq j\leq t\}. 
$$
Note that 
$$ 
h_{\tp}(\mathcal{X}(t))\geq h_{\tp}(\mathcal{H}(t))\quad \text{ for all }t\geq 0.
$$
Let $\mathcal{R}_0$ be the set of all periodic points in $\mathcal{X}_3$.  Define 
$$
\beta_0=(01)^{100k}=01010101\dots0101
$$
where $k$ is sufficiently large such that
    \begin{enumerate}
    \item the word $(01)^k$ does not occur as a subword of $y$; 
    \item $h_{\tp}(\mathcal{H}_3(0))>(1-1/4)\cdot h_{\tp}(\mathcal{H}_3)$ (by Lemma~\ref{lem:lind}, this is possible).
    \end{enumerate}
Let $\mathcal{R}_1$ be the set of all periodic points in $\mathcal{X}_3(0)$. 
Note that none of the points in $\mathcal{R}_1$ contain either $11$ or $1001$ as a subword, and all of them contain $000$ as a subword (the only periodic point in $\mathcal{X}_3$ that does not contain $000$ as a subword is the point $\dots0101010101\dots$ which is forbidden by the word $\beta_0$ in $\mathcal{X}_3(0)$).  
Among the elements of $\mathcal{R}_1$, choose a periodic point $r_1$  that has the smallest possible minimal period.  Choose a word $x_1$ that has $000$ as a prefix and is such that $r_1=\dots x_1x_1x_1x_1x_1\dots$ is (a shift of) the bi-infinite concatenation of $r_1$.  Choosing an integer $s_1$ sufficiently large such that when we take $\beta_1:=x_1^{100s_1}$, we  then  have that 
    \begin{enumerate}
    \item $x_1^{s_1}$ does not occur as a subword of $y$; 
    \item $h_{\tp}(\mathcal{H}_3(1))>(1-1/4-1/16)\cdot h_{\tp}(\mathcal{H}_3)$ (by Lemma~\ref{lem:lind}, this is possible). 
    \end{enumerate} 
Let $\mathcal{R}_2$ be the set of all periodic points in $\mathcal{X}_3(1)$.  We continue inductively and assume that we have  defined words $\beta_0,\dots,\beta_t$ that are each periodic, have increasing lengths, are such that $\beta_i=x_i^{100s_i}$ for some word $x_i$ that has $000$ as a prefix and $s_i$ large enough that $x_i^{s_i}$ does not occur as a subword of $y$ (for all $1\leq i\leq t$), and are such that for all $i=1,\dots,t$ we have 
    $$
    h_{\tp}(\mathcal{H}_3(i))>\left(1-\sum_{j=0}^i\frac{1}{4^{j+1}}\right)\cdot h_{\tp}(\mathcal{H}_3). 
    $$
Let $\mathcal{R}_{t+1}$ be the set of periodic points in $\mathcal{X}_3(t)$.  Note that $\mathcal{R}_{t+1}$ is nonempty because $\mathcal{X}_3(t)$ contains the positive entropy subshift of finite type $\mathcal{H}_3(t)$.  Among the points in $\mathcal{R}_{t+1}$, let $r_{t+1}$ be a periodic point with the smallest possible minimal period.  Find a word $x_{t+1}$ that has $000$ as a prefix and is such that $r_{t+1}=\dots x_{t+1}x_{t+1}x_{t+1}x_{t+1}x_{t+1}\dots$ is (a shift of) the bi-infinite concatenation of $r_{t+1}$.  Choose  an integer $s_{t+1}$ sufficiently large such that if $\beta_{t+1}:=x_{t+1}^{100s_{t+1}}$, then $\beta_{t+1}$ is longer than $\beta_t$ and 
    \begin{enumerate}
    \item $x_{t+1}^{s_{t+1}}$ does not occur as a subword of $y$;
    \item $h_{\tp}(\mathcal{H}_3(t+1))>\left(1-\sum_{j=0}^{t+1}\frac{1}{4^{j+1}}\right)\cdot h_{\tp}(\mathcal{H}_3)$ (by Lemma~\ref{lem:lind}, this is possible).  
    \end{enumerate}
We further require that  $s_{t+1}$ has an additional property that allows us to show  that the shift we build  is period stable.  Let $(\mathcal{Y}_n(t))_{n\in\N}$ be the SFT cover of $\mathcal{X}_3(t)$.  For any fixed $n\in\N$, the shift $\mathcal{Y}_n(t)$ contains many periodic points that are not in $\mathcal{X}_3(t)$.  However, for any periodic point $p$, either $p\in\mathcal{R}_{t+1}$ or there exists $n\in\N$ such that $p\notin\mathcal{Y}_n(t)$.  Since $r_{t+1}$ has the smallest possible minimal period among all periodic points in $\mathcal{R}_{t+1}$, there exists $N_{t+1}$ such that for all $n\geq N_{t+1}$, 
all periodic points in $\mathcal{Y}_n(t)$ that do not lie in $\mathcal{R}_{t+1}$ have minimal period strictly larger than the minimal period of $r_{t+1}$.  When choosing the integer $s_{t+1}$, we require  that the length of $\beta_{t+1}=x_{t+1}^{100s_{t+1}}$ is longer than $N_{t+1}+t+1$.  Define $\gamma_{t+1}$ to be $\beta_{t+1}$ with its first letter removed, and $\zeta_{t+1}$ to be $\beta_{t+1}$ with its last letter removed.  If $(\mathcal{Y}_n(t+1))_{n\in\N}$ is the SFT cover of $\mathcal{X}_3(t+1)$, we claim that the periodic points of minimal possible period in $\mathcal{Y}_{N_{t+1}}(t+1)$ coincide with those of minimal possible period in $\mathcal{Y}_{N_{t+1}+t}(t+1)$.  To see this, note that any periodic point in $\mathcal{Y}_{N_{t+1}}(t+1)$ of minimal possible period that does not contain $\beta_{t+1}$ as a subword, is in $\mathcal{Y}_{N_{t+1}+t}(t+1)$.  If we can show that both $\gamma_{t+1}$ and $\zeta_{t+1}$ are in $\mathcal{L}(\mathcal{Y}_{N_{t+1}+t}(t+1))$, then any periodic point $\mathcal{Y}_{N_t}(t+1)$ that does contain $\beta_{t+1}$ as a subword is also in $\mathcal{Y}_{N_{t+1}+t}(t+1)$ because any forbidden word that occurs it must contain $\beta_{t+1}$ as a subword, but $|\beta_{t+1}|>N_{t+1}+t$.  To see that both $\gamma_{t+1}$ and $\zeta_{t+1}$ are in $\mathcal{L}(\mathcal{Y}_{N_{t+1}+t}(t+1))$, let $A$ be a semi-infinite ray in $y$ that emanates to the left and ends with a $1$, and let $\Omega$ be a semi-infinite ray in $y$ that emanates to the right and starts with a $y$ (where $y\in Y$ is the element of the Sturmian shift $Y$ we fixed at the start of the construction).  Notice that $A\gamma_{t+1}\Omega\in\mathcal{Y}_{N_{t+1}+t}(t+1)$ (and similarly $A\gamma_{t+1}\Omega\in\mathcal{Y}_{N_{t+1}+t}(t+1)$) because it does not contain $11$ as a subword, it contains $1001$ as a subword at most twice (overlapping $A$ and $\gamma_{t+1}$ or overlapping $\gamma_{t+1}$ and $\Omega$), and all occurrences of $000$ as a subword occur within $\gamma_{t+1}$, so it does not contain any of the already forbidden words as a subword.
Finally, we require $s_{t+1}$  be sufficiently large such  that $|\beta_{t+1}|>2|\beta_t|$.

Inductively this procedure defines the words $\beta_t$ for all $t\geq 1$.  We define 
\begin{equation}
    \label{def:X4}
\mathcal{X}_4:=\bigcap_{t=1}^{\infty}\mathcal{X}_3(t)
\end{equation} 
and claim that $\mathcal{X}_4$ is period stable, is not language stable, and does not have any periodic points.

\subsubsection{Properties of $\mathcal{X}_4$} 
We check the properties of the shift $\mathcal{X}_4$ defined in~\eqref{def:X4}, showing that it is a nonempty, period stable but not language stable shift that has no periodic points.  It then follows from period stability that  $\mathcal{X}_4$ 
has a characteristic measure, and the existence of this measure does not follow either from other criteria, such as language stability or existence of a periodic point. Moreover, as $\mathcal X_4$ has no periodic point, the characteristic measure we produce is non-atomic. 

\begin{lemma}\label{lem:nonemtpy}
$\mathcal{X}_4$ is nonempty.
\end{lemma}
\begin{proof}
It follows from the definition~\eqref{def:X4} of $\mathcal X_4$ that this shift is the intersection of the nested sequence of shifts $\mathcal{X}_3(0)\supset\mathcal{X}_3(1)\supset\mathcal{X}_3(2)\supset\dots$.  Furthermore, it follows from the construction that 
$$ 
h_\tp(\mathcal{X}_3(t))\geq h_\tp(\mathcal{H}_3(t))>\left(1-\sum_{j=0}^t\frac{1}{4^{j+1}}\right)\cdot h_{\tp}(\mathcal{H}_3)>\frac{2}{3}\cdot\log_2(1.512).
$$
By upper semi-continuity of entropy, this means $h_{\tp}(\mathcal{X}_4)\geq\frac{2}{3}\cdot\log_2(1.512)>0$, and in particular  $\mathcal{X}_4$ is nonempty.
\end{proof}

\begin{lemma}\label{lem:no-periodic}
$\mathcal{X}_4$ has no periodic points.
\end{lemma}
\begin{proof}
This is immediate from the construction of the sets $\mathcal{W}_1$, $\mathcal{W}_2^{11}$, $\mathcal{W}_2^{00}$, and $\mathcal{W}_2^{000}$ which were built specifically to add forbidden words that eliminated all periodic points.
\end{proof} 

To show that $\mathcal{X}_4$ is not language stable, we introduce some notation.  
For any $m\in\Z$,  let $L_m\in\{0,1\}^{\{n\in\Z\colon n\leq m\}}$ be the $\{0,1\}$-valued coloring of the left-infinite ray $\{n\in\Z\colon n\leq m\}$ given by 
$$
L_m(n):=y(1+n-m).
$$
In other words, $L_m$ is the left-infinite ray obtained by restricting $y$ to the set $\{n\in\Z\colon n\leq1\}$ and then shifting the ray such  that its rightmost edge is at $m$ instead of $1$.  Similarly,  let $R_m\in\{0,1\}^{\{n\in\Z\colon n\geq m\}}$ be the $\{0,1\}$-valued coloring of the right-infinite ray $\{n\in\Z\colon n\geq m\}$ given by 
$$
R_m(n):=y(1+n-m)
$$
We have that $L_m(m)=R_m(m)=0$, because $y(1)=0$.

For $w\in\{0,1\}^*$, we  index the letters of $w$ starting at $1$, meaning we consider $w$ to be a function $w\colon\{1,2,\dots,|w|\}\to\{0,1\}$.  Set $LwR\colon\Z\to\{0,1\}$ to be the function 
$$ 
LwR(i)=\left\{\begin{tabular}{cl} 
$L_0(i)$ & if $i\leq0$; \\ 
$w(i)$ & if $1\leq i\leq|w|$; \\ 
$R_{|w|+1}(i)$ & if $|w|+1\leq i$. 
\end{tabular}\right. 
$$ 
In other words, $LwR$ is the $\{0,1\}$-valued coloring of $\Z$ obtained by concatenating the left-infinite ray $L_0$ with $w$ and $R_{|w|+1}$.  By construction, $LwR(0)=LwR(|w|+1)=0$ and the restriction of $LwR$ to the set $\{1,2,\dots,|w|\}$ is $w$.  For clarity, we introduce terminology to reflect how a word occurs as a subword of a concaentation of words. 
If $u$ is a subword of $LwR$ that occurs at coordinates entirely within the set $\{i\colon i\leq0\}$, we say $u$ {\em occurs in $L$}.  If $u$ occurs at coordinates entirely in the set $\{i\colon i\geq|w|+1\}$, we  say $u$ {\em occurs in $R$}.  If $u$ occurs at coordinates entirely in the set $\{1,2,\dots,|w|\}$, we  say $u$ {\em occurs in $w$}.

\begin{lemma}\label{lem:not-language-stable}
$\mathcal{X}_4$ is not language stable. 
\end{lemma}
As a remark, this is the point at which we use the words defined in $\mathcal{W}_1$.
\begin{proof} 
Letting $(\mathcal{S}_n)_{n\in\N}$ be the SFT cover of $\mathcal{X}_4$, we show that the set 
$$ 
\{n\in\N\colon\mathcal{S}_n\neq\mathcal{S}_{n+1}\} 
$$ 
is syndetic in $\N$.  

Fix $n\in\N$.  The word $u_n\in\mathcal{W}_1$ is forbidden in $\mathcal{X}_4$.  Let $a_n$ be the word $u_n$ with its rightmost letter removed (thus  $|a_n|=|u_n|-1$).  Let $b_n$ be the word $u_n$ with its leftmost letter removed (again $|b_n|=|u_n|-1$).  We  show that $a_n,b_n\in\mathcal{L}(\mathcal{X}_4)$.  It follows from this that $u_n\in\mathcal{L}(\mathcal{S}_{|a_n|})=\mathcal{L}(\mathcal{S}_{|u_n|-1})$ but $u_n\notin\mathcal{L}(\mathcal{S}_{|u_n|}$); hence $\mathcal{S}_{|u_n|-1}\neq\mathcal{S}_{|u_n|}$.  Since $\{|u_n|\colon n\in\N\}$ is a syndetic subset of $\N$, it follows that $\mathcal{X}_4$ is not language stable.  
Thus we are left with showing that  $a_n,b_n\in\mathcal{L}(\mathcal{X}_4)$.

To show that $a_n\in\mathcal{L}(\mathcal{X}_4)$, we prove  that $La_nR\in\mathcal{X}_4$ (a similar argument shows that $Lb_nR\in\mathcal{X}_4$ and so we also have that $b_n\in\mathcal{L}(\mathcal{X}_4)$).  To check that $La_nR\in\mathcal{X}_4$, we need to check that no element of $\mathcal{W}_1\cup\mathcal{W}_2^{11}\cup\mathcal{W}_2^{000}\cup\mathcal{W}_2^{00}$ occurs as a subword of $La_nR$.

\subsubsection*{Claim 1: If $w\in\mathcal{W}_1$, then $w$ does not occur as a subword of $La_nR$.
} 
 For contradiction, suppose $w$ occurs as a subword of $La_nR$.  Since $w\in\mathcal{W}_1$, this word has $111$ as both a prefix and a suffix.  
 Since $11$ does not occur as a subword of $y$, we know $11$ does not occur in $L$ and also does not occur in $R$.  
 By construction $(La_nR)(0)=(La_nR)(|w|+1)=0$, and so any location where $11$ occurs in $La_nR$ must be entirely within $a_n$ (in other words it cannot partially overlap $L$ and $a_n$ or $a_n$ and $R$).  Similarly, the suffix $11$ of $w$ must occur in $a_n$, meaning $w$ is a subword of $a_n$.  Therefore $|w|\leq|a_n|$.  But $w=u_m$ for some $m\in\N$ and $|a_n|=|u_n|-1$, and so $m<n$.  That means that $|w|\leq|u_n|-4=|a_n|-3$.  But $11$ only occurs as a subword of $u_n$ in the positions~\eqref{pos1},~\eqref{pos2}, and~\eqref{pos3} in the construction in Section~\ref{subsec:W1}.  Because $|w|\geq6$, the prefix $11$ of $w$ must occur within the prefix $111$ of $a_n$ and the suffix $11$ of $w$ must occur within the suffix $11$ (if $u_n$ ends with $0111$) or the suffix $111$ (if $u_n$ ends with $01111$) of $a_n$.  Therefore $|w|\geq|a_n|-2$; a contradiction. 

\subsubsection*{Claim 2: If $w\in\mathcal{W}_2^{11}$, then $w$ does not occur in $La_nR$.}
 As in the proof of Claim 1, since $w$ begins and ends with $11$, any occurrence of $w$ in $La_nR$ must be within $a_n$. 
 Therefore, since $a_n$ is a subword of $u_n$, the word $w$ must occur as a subword of $u_n$.  Recall that the words in $\mathcal{W}_2^{11}$ are periodic, begin with the word $11$, and repeat their period at least $100$ times.  Therefore, if $w$ is a subword of $u_n$, there must be at least $100$ separate locations in $u_n$ where $11$ occurs as a subword.  But the word $11$ only occurs as a subword of $u_n$ in positions~\eqref{pos1},~\eqref{pos2}, and~\eqref{pos3} in the construction in Section~\ref{subsec:W1}.  Therefore $w$ cannot occur as a subword of $u_n$ and so also can not occur as a subword of $La_nR$.

\subsubsection*{Claim 3: If $w\in\mathcal{W}_2^{00}$, then $w$ does not occur in $La_nR$.}
Again we proceed by contradiction and assume that $w$ does occur in $La_nR$.  Recall from the construction of $w$ in Section~\ref{sec:zerozero} that there is a word $\alpha_i$ starting with $00$ and a parameter $\ell_i\in\N$ sufficiently large such that $\alpha_i^{\ell_i}$ does not occur as a subword of $y$.  Then $w=\alpha_i^{100\ell_i}00$.  Since $\alpha_i^{\ell_i}$ is not a subword of $y$, it is also not a subword of $L$ or $R$.  
Thus if $w$ occurs in $La_nR$ the longest prefix of $w$ that appears in $L$ has length less than $|w|/100$ and the longest suffix of $w$ that appears in $R$ has length less than $|w|/100$.  Therefore, a subword of $w$ of length at least $98|w|/100$ occurs as a subword of $a_n$.  As $a_n$ is a subword of $u_n$, a subword of $w$ of length at least $98|w|/100$ appears as a subword of $u_n$.  
Recall that $u_n$ has a prefix $111$ and a suffix $111$, but if both of these are removed from $u_n$, the remaining subword is a subword of $y$.  Since there is a subword of $w$ of length at least $98|w|/100$ that occurs as a subword of $u_n$, then a subword of $w$ of length at least $92|w|/100$ appears as a subword of $y$.  In particular, this means $\alpha_i^{\ell_i}$ appears as a subword of $y$; a contradiction.

\subsubsection*{Claim 4: If $w\in\mathcal{W}_2^{000}$, then $w$ does not occur in $La_nR$.}
The word $000$ does not occur as a subword of $y$, and so it is not a subword of either $L$ or $R$.  Since $a_n$ begins and ends with the word $11$, there is no occurrence of $000$ in $La_nR$ that partially overlaps $L$ and $a_n$ or $a_n$ and $R$.  
Finally, $000$ does not occur as a subword of $a_n$ because $a_n$ is a subword of $u_n$ and $u_n$ decomposes as $u_n=111\gamma111$ where $\gamma$ is a subword of $y$.  Therefore $000$ does not occur as a subword of $La_nR$.  
Thus if $000$ is a subword of $w$, then $w$ cannot be a subword of $La_nR$.  The only word in $\mathcal{W}_2^{000}$ that does not contain $000$ as a subword is the word $\beta_0=(01)^{100k}$, where $k\in\N$ is so large that $(01)^k$ does not occur as a subword of $y$.  Analogous to the argument for the words in $\mathcal{W}_2^{00}$, if $\beta_0$ occurred in $La_nR$, then a subword of $\beta_0$ of length at least $98|\beta_0|/100$ would have to occur as a subword of $u_n$ and hence a subword of $\beta_0$ of length at least $92|\beta_0|/100$ would be a subword of $y$.  In particular, $(01)^k$ would be a subword of $y$; a contradiction.

As these four cases cover all of the elements $\mathcal{W}_1\cup\mathcal{W}_2^{11}\cup\mathcal{W}_2^{000}\cup\mathcal{W}_2^{00}$, this completes the proof.
\end{proof} 

\begin{lemma}\label{lem:period-stable}
$\mathcal{X}_4$ is period stable.
\end{lemma}
\begin{proof}
Let $\{X_n\}_{n=1}^{\infty}$ be the SFT cover of $\mathcal{X}_4$.  Recall that 
$$
\mathcal{X}_4=\bigcap_{t=0}^{\infty}\mathcal{X}_3(t) 
$$
where $\mathcal{X}_3(0)\supset\mathcal{X}_3(1)\supset\dots\supset\mathcal{X}_3(t)\supset\dots$.  In the notation of Section~\ref{sec:zerozerozero}, recall that  for any $t$, $(\mathcal{Y}_n(t))_{n\in\N}$ is the SFT cover of $\mathcal{X}_3(t)$ and that $N_t$ is a parameter constructed such that the collection of periodic points of minimal period in $\mathcal{Y}_{N_t}(t)$ concides with those of minimal period in $\mathcal{Y}_{N_t+t}(t)$.  
Further, for any $t,k\in\N$, the shift $\mathcal{X}_3(t+k)$ is obtained from $\mathcal{X}_3(t)$ by forbidding the words $\beta_{t+1},\beta_{t+2},\dots,\beta_{t+k}$ and by construction, $|\beta_{t+j}|>2^j|\beta_t|$ for all $j>0$.  For each $t\in\N$, let $M_t$ be the minimal period of any periodic point in $\mathcal{Y}_{N_t}$.  Note that $M_t$ is also the minimal period of any periodic point in $\mathcal{X}_3(t-1)$ and that collection of periodic points of period $M_t$ in $\mathcal{X}_3(t)$ is strictly smaller than that of $\mathcal{X}_3(t-1)$ (because $\beta_t$ is specifically constructed to eliminate one of these periodic points when defining $\mathcal{X}_3(t)$). 
Therefore, the sequence $\{M_t\}_{t=1}^{\infty}$ is non-decreasing and tends to infinity.  For any value of $t$ with  $M_{t+1}>M_t$, all periodic points of period $M_t$ in $\mathcal{X}_3(t-1)$ contain $\beta_t$ as a subword.  Let $\{t_r\}_{r=1}^{\infty}$ be a subsequence for which $M_{t_r}<M_{t_r+1}$ for all $r\geq 1$.  Then for any $r$, we know 
    \begin{enumerate}
    \item the minimum possible period of any periodic point in $\mathcal{X}_3(t_r-1)$ is $M_{t_r}$; 
    \item all periodic points of period $M_{t_r}$ in $\mathcal{X}_3(t_r-1)$ contain $\beta_{t_r}$ as a subword; 
    \item all periodic points of period $M_{t_r}$ in $\mathcal{X}_3(t_r-1)$ are in $\mathcal{Y}_{N_{t_r}}(t_r)$ and in $\mathcal{Y}_{N_{t_r}+t_r}(t_r)$; 
    \item all periodic points in $\mathcal{X}_3(t_r)$ have period strictly larger than $M_{t_r}$; 
    \item $|\beta_{t_r+j}|>2^j|\beta_{t_r}|$ for all $j>0$.
    \end{enumerate}
Let $p$ be a periodic point of period $M_{t_r}$ in $\mathcal{X}_3(t_r-1)$.  While $p$ must contain $\beta_{t_r}$ as a subword, we claim that it does not contain $\beta_{t_r+j}$ as a subword for any $j>0$.  Since $|\beta_{t_r}|>N_{t_r}+t_r$, it follows from this and the fact that $p\in\mathcal{Y}_{N_{t_r}+t_r}(t_r)$ that $p\in X_{N_{t_r}}$ and $p\in X_{N_{t_r}+t_r}$.  On the other hand, every periodic point of minimum period $M_{t_r}$ in $X_{N_{t_r}}$ must be in $\mathcal{X}_3(t_r-1)$ because $\mathcal{Y}_{N_{t_r}}(t_r)$ is a subset of $X_{N_{t_r}}$.  So, provided we can establish the claim, it follows that the periodic points of minimum possible period in $X_{N_{t_r}}$ coincide with those in $X_{N_{t_r}+t_r}$.  Since $\{t_r\}_{r=1}^{\infty}$ in strictly increasing, it follows that $\mathcal{X}_4$ is period stable.  
Thus we are left with showing that $p$ does not contain any of the words $\beta_{t_r+j}$ for $j>0$.  As noted, $p$ does contain $\beta_{t_r}$ as a subword.  Moreover, $|\beta_{t_r}|>|M_{t_r}|$ (the minimal period of $p$).  By construction, $|\beta_{t_r+j}|>2^j|\beta_{t_r}|$ and so if $p$ also contained $\beta_{t_r+j}$ as a subword, then $\beta_{t_r+j}$ would itself contain $\beta_{t_r}$ as a subword.  But $\beta_{t_r+j}$ is a word in the language of the subshift $\mathcal{X}_3(t_r+j)$ and $\beta_{t_r}$ is a forbidden word in $\mathcal{X}_3(t_r+j)$; a contradiction.  We conclude that $\beta_{t_r+j}$ cannot be a subword of $p$ for any $j>0$.
\end{proof}

\subsection{Construction of such a system with nontrivial automorphism group}
\label{sec:nontrivial}
In Section~\ref{sec:construction} we constructed an example of a symbolic system, $(X,T)$, with the following properties: 
    \begin{enumerate}[label=(\roman*)]
    \item 
    \label{item:new-one}
    $(X,T)$ is period stable; 
    \item $(X,T)$ is not language stable; 
    \item 
     \label{item:new-three} 
     $(X,T)$ has no periodic points. 
    \end{enumerate}
Interest in $(X,T)$ is that by Theorem~\ref{th:periodic} it carries a characteristic measure that is not seen to exist for other reasons,  such as being language stable or having a periodic point.  However, it is not clear what the automorphism group of $(X,T)$ is, and for example, if $\Aut(X)$ were to be amenable, the existence of a characteristic measure follows from the Krylov-Bogolyubov Theorem.  To address this possibility,  we show how to use $(X,T)$ to construct a  system $(Z,T)$ that retains properties~\eqref{item:new-one}--\eqref{item:new-three} of $(X,T)$ and furthermore has a non-amenable automorphism group.  Specifically, if $(X, T)$ denotes the system constructed in Section~\ref{sec:construction} (as defined in~\eqref{def:X4}) and if $(Y, T)$ denotes the full shift on two symbols, we show that the subshift $Z=X\times Y$ has the following properties: 
    \begin{enumerate}[label=(\roman*)]
    \item $(Z,T)$ is period stable; 
    \item $(Z,T)$ is not language stable; 
    \item $(Z,T)$ has no periodic points; 
    \item $\Aut(Z)$ contains the free group on two generators (and so in particular is not amenable).
    \end{enumerate}

  \subsubsection{(Z,T) has no periodic points} The system $(X, T)$ has no periodic points by construction, and since any periodic point in $Z$ projects to a periodic point in $X$, there are no periodic points in $(Z, T)$.  

  \subsubsection{(Z,T) is not language stable} Let $\{X_n\}_{n=1}^\infty$ denote the SFT cover of $X$ and 
  $\{Z_n\}_{n=1}^\infty$ denote the SFT cover of $Z$. As the full shift $(Y, T)$ has no forbidden words, it follows that 
  $Z_n = X_n\times Y$ for all $n\in\N$.  
  As the system $(X, T)$ is not language stable, this then implies that $(Z,T)$ is not language stable. 

  \subsubsection{(Z,T) is period stable} Let $p_n$ denote the smallest period of any periodic point in the $n$-th subshift $(X_n, T)$ in the SFT cover of $(X, T)$ and let 
  $P_n(X)$ denote the set of periodic points of period $p_n$ in $X_n$. 
 Since $(X, T)$ is period stable, given  $k\geq 1$ there is some $n\in\N$ such that $P_n(X) = P_{n+k}(X)$.  Then if $x\in P_n$ and $y\in Y$ is a fixed point,  the point $(x, y)\in Z$ has period $p_n$ and so $Z_n$ has periodic points of period $p_n$. Moreover, no periodic point $(x', y')\in Z_n$ can have 
 period lower than $p_n$, as otherwise $x'\in X_n$ would have a period smaller than $p_n$.  Thus $p_n$ is the smallest period of any periodic point in $Z_n$.  
 
  If $(x,y)\in Z_n$ is some periodic point of period $p_n$, 
  then $x\in X_n$ and $y\in Y$ are periodic, and the period of each divides $p_n$. Since $p_n$ is the least period of any periodic point in $X_n$, it follows that $x$ has period $p_n$.  Thus the set of periodic points in $Z_n$ with period $p_n$ is equal to the set of all points  $(x, y)$  with $x\in P_n(X)$ and $y\in Y$ having period dividing $p_n$. 
  Similarly, the set of periodic points of period $p_n$ in $Z_{n+k}$ is also the set of all points  $(x, y)$  with $x\in P_{n+k}(X)= P_n(X)$ and $y\in Y$ having period dividing $p_n$.  (We note that there is no further restriction on the second coordinate as $Y$ is a full shift.) 
  Thus if follows that $(Z, T)$ is period stable.

  \subsubsection{Aut(Z) contains the free group on two generators} We have  that $\Aut(Z)$ contains $\Aut(X)\times \Aut(Y)$ as a subgroup.  Since $\Aut(Y)$ is contains the free group on two generators, it follows that $\Aut(Z)$ also does (and so is not amenable).

\end{document}